\documentclass[12pt, reqno]{amsart}
\usepackage{amsmath,amssymb,amsthm}
\usepackage{amsmath, amssymb}
\usepackage{amsthm, amsfonts, mathrsfs}
\usepackage{mathptmx}
\usepackage{fullpage}
\usepackage{amsfonts,graphicx}
\numberwithin{equation}{section}
\usepackage[colorlinks=true, pdfstartview=FitV, linkcolor=blue, citecolor=blue, urlcolor=blue]{hyperref}

\allowdisplaybreaks

\numberwithin{equation}{section}

\def\div{ \hbox{\rm div}\,  }

\numberwithin{equation}{section}
\newtheorem{theorem}{Theorem}[section]

\newtheorem{lemma}[theorem]{Lemma}

\numberwithin{equation}{section}

\theoremstyle{remark}
\newtheorem{remark}[theorem]{Remark}
\def\T{ \mathbb{T} }

\newcommand{\R}{{\mathbb{R}}}
\newcommand{\Dv}{{\rm div}}
\def\div{ \hbox{\rm div}\,  }
\def\u{ \mathbf{u} }

\def\b{ \mathbf{B} }
\def\H{ \mathbf{H} }
\def\T{ \mathbb{T} }

\newcommand{\norm}[2]{\left\lVert #1 \right\rVert_{#2}}

\begin{document}

\title[3D Compressible MHD equations]{Stabilization by a background magnetic field: global well-posedness of the compressible isentropic ideal MHD equations with velocity damping}

\author[L. Qiao]{Liening  Qiao }
\address[L. Qiao]{ School of Mathematics and  Statistics, Shandong University of
Technology,  Zibo 255049,  Shandong Province, China}
%\address[L. Qiao]{$^{2}$ School of Mathematical Sciences, Xiamen University, Xiamen 361005, Fujian, China}
\email{lnqiao2026@163.com}

\author[J. Wu]{Jiahong Wu}
\address[J. Wu]{Department of Mathematics, University of Notre
	Dame, Notre Dame, IN 46556, USA } \email{jwu29@nd.edu}
 \author[F. Xu]{Fuyi  Xu$^{\dag}$}
\address[F. Xu]{$^{\dag}$ School of Mathematics and  Statistics, Shandong University of
Technology,  Zibo 255049,  Shandong Province, China} \email{zbxufuyi@163.com (Corresponding author)}

\author[X. Zhai]{Xiaoping Zhai}
\address[X. Zhai]{School of Mathematics and Statistics, Guangdong University of Technology,
	Guangzhou, 510520, China.} \email{pingxiaozhai@163.com }

\date{}

\begin{abstract} In the present paper,
we study the Cauchy problem for the triple-dimensional isentropic
compressible ideal (inviscid and non-resistive) magnetohydrodynamic
equations with velocity damping on the periodic torus $\mathbb{T}^3$.
The system admits a steady equilibrium consisting of a constant density
$\bar{\rho}$ and a uniform background magnetic field $\omega\in\mathbb{R}^3$.
We prove that this equilibrium is nonlinearly stable. More precisely,
we show that if the initial data
are a sufficiently small perturbation of $(\bar{\rho},\mathbf{0},\omega)$
 in a high-order Sobolev space, and if
$\omega$ satisfies a Diophantine condition, then the system admits a unique global smooth solution.  Moreover, the
perturbations decay
algebraically in time. To the best of our knowledge, this is the first global well-posedness
result for the 3D isentropic compressible ideal MHD
system.  The proof reveals a hidden dissipation mechanism: although
neither the density equation nor the magnetic field equation contains
explicit diffusion or damping, the coupling between the velocity and
the magnetic field through the background field $\omega$, combined with
a Diophantine--Poincar\'{e} inequality, generates effective dissipation
for both the density perturbation and the magnetic field perturbation,
which together with the velocity damping yields global regularity and
time decay.
\end{abstract}

\maketitle

\tableofcontents

%% ============================================================
%%  REVISED INTRODUCTION
%% ============================================================

\section{Introduction and Main Result}
\label{sec1}

%% ------------------------------------------------------------------
\subsection{Physical background and motivation}
%% ------------------------------------------------------------------

The dynamics of compressible inviscid fluids are governed by the
Euler equations for conservation of mass and momentum.  In many
physically relevant settings the fluid moves through a porous medium,
or is subject to a linear drag force proportional to the local
velocity, a mechanism referred to as velocity damping.
The resulting system,
\begin{equation}\label{Euler-damping}
	\left\{
	\begin{aligned}
		&\partial_t \rho + \mathrm{div}(\rho \mathbf{u}) = 0, \\
		&\partial_t(\rho \mathbf{u}) + \mathrm{div}(\rho \mathbf{u}\otimes\mathbf{u})
		+ \kappa\rho\mathbf{u} + \nabla P(\rho) = 0,
	\end{aligned}
	\right.
\end{equation}
arises in the modelling of gas flow through porous media
(Darcy's law \cite{H-L,H-D}), the propagation of acoustic waves in
dissipative media, and the large-time dynamics of cosmological matter
distributions \cite{STW}.  The damping term $\kappa\rho\mathbf{u}$
($\kappa>0$) plays a fundamentally regularizing role. It dissipates
kinetic energy and prevents the formation of shocks that would otherwise
develop from smooth data in finite time for the undamped Euler system ($\kappa=0$).
Seminal works by Hsiao--Luo \cite{H-L}, Hsiao--Serre \cite{H-D},
Huang--Pan \cite{H-P}, and Huang--Marcati--Pan \cite{H-M-P} established
convergence of solutions toward the nonlinear diffusion wave described by
the porous medium equation in the one and multi-dimensional settings.
For smooth small-data solutions on the torus, Sideris--Thomases--Wang
\cite{STW} proved global existence and exponential decay for the
three-dimensional problem.  Thus, on the periodic domain $\mathbb{T}^3$,
 system \eqref{Euler-damping} is globally well-posed for small
perturbations of a constant state, and the damping is solely responsible
for this regularity.

\medskip
Electrically conducting compressible fluids such as plasmas, liquid metals,
ionised gases are described by the magnetohydrodynamic (MHD)
equations, which couple the compressible Euler equations with Maxwell's
equations of electromagnetism via the Lorentz force.  In the
ideal regime, where both the kinematic
viscosity $\mu, \lambda$ and the magnetic resistivity $\nu$ vanish, the
governing system with velocity damping on $\mathbb{T}^3$ is
\begin{equation}\label{MHD1111}
	\left\{
	\begin{aligned}
		&\partial_t \rho + \mathrm{div}(\rho \mathbf{u}) = 0, \\
		&\partial_t(\rho \mathbf{u}) + \mathrm{div}(\rho\mathbf{u}\otimes\mathbf{u})
		+ \kappa\rho\mathbf{u} + \nabla P
		= \mathbf{H}\cdot\nabla\mathbf{H} - \tfrac{1}{2}\nabla|\mathbf{H}|^2, \\
		&\partial_t\mathbf{H} + \mathbf{u}\cdot\nabla\mathbf{H}
		= \mathbf{H}\cdot\nabla\mathbf{u} - \mathbf{H}\,\mathrm{div}\,\mathbf{u}, \\
		&\mathrm{div}\,\mathbf{H} = 0, \\
		&(\rho,\mathbf{u},\mathbf{H})|_{t=0} = (\rho_0,\mathbf{u}_0,\mathbf{H}_0).
	\end{aligned}
	\right.
\end{equation}
Here $\rho$, $\mathbf{u}$, and $\mathbf{H}$ denote density, velocity, and
magnetic field, $\kappa>0$ is the damping coefficient, and the isentropic
pressure $P=P(\rho)$ is a smooth, strictly increasing function.  The
right-hand side of the momentum equation is the Lorentz body force
generated by the magnetic field, while the induction equation for
$\mathbf{H}$ encodes flux freezing property of ideal
(perfectly conducting) MHD, namely in the absence of resistivity the magnetic
field lines are transported with the fluid.

\medskip
System \eqref{MHD1111} models a broad class of physical phenomena including
magnetised plasma dynamics in fusion devices
\cite{Cabannes-H,Landau-Lifshitz}, astrophysical accretion discs
\cite{Polovin-Demutskii}, liquid-metal flows in metallurgical
engineering, and geophysical dynamo processes in planetary interiors.
Even though \eqref{MHD1111} is the most transparent,
member of the compressible MHD hierarchy, the absence of viscosity and resistivity renders  \eqref{MHD1111} the most
challenging mathematically.

%% ------------------------------------------------------------------
\subsection{Equilibrium state and the stability problem}
%% ------------------------------------------------------------------

A fundamental and physically natural \emph{steady state} of
\eqref{MHD1111} is the configuration
\begin{equation}\label{equilibrium}
	(\rho,\mathbf{u},\mathbf{H}) \equiv (\bar\rho,\,\mathbf{0},\,\omega),
\end{equation}
where $\bar\rho>0$ is a constant reference density and
$\omega\in\mathbb{R}^3\setminus\{0\}$ is a constant background
magnetic field.  This state represents a perfectly conducting fluid at
rest, threaded by a uniform magnetic field.  It is easily verified
that \eqref{equilibrium} satisfies \eqref{MHD1111} identically.

\medskip
The central question of this paper is whether the
equilibrium \eqref{equilibrium} is nonlinearly stable in the class
of smooth solutions. Given initial data that are a small perturbation of
$(\bar\rho,\mathbf{0},\omega)$ in a high-order Sobolev space, does a
unique global smooth solution exist and remain close to the equilibrium
for all time?  In the absence of viscosity or magnetic diffusion, this
question is genuinely open in general.  The present paper provides an
affirmative answer by showing that the combination of velocity damping
and a background magnetic field satisfying a Diophantine condition
(see \eqref{Diophantine0} below) is sufficient to guarantee global
regularity and algebraic decay.

%% ------------------------------------------------------------------
\subsection{Perturbation formulation}
%% ------------------------------------------------------------------

Setting $\kappa=1$ for simplicity (the value of $\kappa$ does not affect
the analysis), we introduce
\begin{equation}\label{perturbation-def}
	a := \rho - \bar\rho, \qquad \mathbf{u}, \qquad \b := \mathbf{H}-\omega.
\end{equation}
Substituting \eqref{perturbation-def} into \eqref{MHD1111}, the perturbed system reads
\begin{equation}\label{MHD1}
	\left\{
	\begin{aligned}
		&\partial_t a + \mathrm{div}(\rho\mathbf{u}) = 0, \\[4pt]
		&\rho(\partial_t\mathbf{u} + \mathbf{u}\cdot\nabla\mathbf{u}) + \rho\mathbf{u}
		+ \nabla P
		= \underbrace{\omega\cdot\nabla\b
			- \nabla(\omega\cdot\b)}_{\text{linear in }\b}
		+ \underbrace{\b\cdot\nabla\b
			- \tfrac{1}{2}\nabla|\b|^2}_{\text{nonlinear}}, \\[4pt]
		&\partial_t\b + \mathbf{u}\cdot\nabla\b
		= \underbrace{\omega\cdot\nabla\mathbf{u} - \omega\,\mathrm{div}\,\mathbf{u}}_{\text{linear in }\mathbf{u}}
		+ \underbrace{\b\cdot\nabla\mathbf{u} - \b\,\mathrm{div}\,\mathbf{u}}_{\text{nonlinear}}, \\[4pt]
		&\mathrm{div}\,\b = 0, \\[4pt]
		&(\rho,\mathbf{u},\b)|_{t=0} = (\rho_0,\mathbf{u}_0,\b_0).
	\end{aligned}
	\right.
\end{equation}
Taking $\bar\rho=1$ for notational convenience, denoting $\beta:=P'(1)>0$,
and separating linear from nonlinear contributions, \eqref{MHD1} can be written as
\begin{equation}\label{LinNon}
	\left\{
	\begin{aligned}
		&\partial_t a + \mathrm{div}\,\mathbf{u} = f_1, \\[3pt]
		&\partial_t\mathbf{u} + \beta\nabla a + \mathbf{u}
		= \omega\cdot\nabla\b - \nabla(\omega\cdot\b) + f_2, \\[3pt]
		&\partial_t\b
		= \omega\cdot\nabla\mathbf{u} - \omega\,\mathrm{div}\,\mathbf{u} + f_3, \\[3pt]
		&\mathrm{div}\,\b = 0,
	\end{aligned}
	\right.
\end{equation}
where the nonlinear remainder terms are
\begin{equation}\label{nonlinear-terms}
	\begin{aligned}
		f_1 &:= -\mathrm{div}(a\mathbf{u}), \\
		f_2 &:= -a\partial_t\mathbf{u} - (P'(\rho)-\beta)\nabla a-a\u
		- \rho\,\mathbf{u}\cdot\nabla\mathbf{u}
		+ \b\cdot\nabla\b - \tfrac{1}{2}\nabla|\b|^2, \\
		f_3 &:= -\mathbf{u}\cdot\nabla\b
		+ \b\cdot\nabla\mathbf{u} - \b\,\mathrm{div}\,\mathbf{u}.
	\end{aligned}
\end{equation}

%% ------------------------------------------------------------------
\subsection{Main difficulties}
%% ------------------------------------------------------------------

The global analysis of \eqref{MHD1} faces several difficulties
that are different from those encountered in the viscous or resistive settings.

\medskip
First, the continuity equation $\partial_t a + \mathrm{div}(\rho\mathbf{u})=0$
contains no diffusive or damping term acting directly on $a$.  In the
compressible viscous setting, the combination of viscous stresses
and the pressure gradient yields an effective viscous flux
$F = \mathrm{div}\,\mathbf{u} - P$ with better regularity properties than
its constituents individually \cite{HS2,DH1}, and this regularity
is used to extract a damping effect for the density.  In the ideal
setting ($\mu=\lambda=0$) this device is entirely unavailable; see
\cite{Dong-Wu-Zhai,Wu-Wu} for discussions of why the effective-flux
method fails here.

\medskip
Second, the induction equation
is a pure transport-stretching equation with no smoothing in $\b$.
High-frequency oscillations in the magnetic field are neither damped nor
diffused.  For the incompressible ideal MHD system, the
Els\"{a}sser-variable transformation $\mathbf{z}^\pm = \mathbf{u}\pm\b$
produces a symmetric hyperbolic system that allows for uniform-in-time energy estimates
\cite{BSS,Cai,HXY}. However,  in the compressible case the divergence of $\mathbf{u}$
destroys this symmetry, and no analogous transformation is available. %Moreover, th emethod applied in \cite{BSS,Cai} is purely hyperbolic (characteristic method) and hence could not be applied in our case.

\medskip
Third,  a naive energy estimate yields
$\frac{d}{dt}\mathcal{E}(t) + \|\mathbf{u}\|^2 \leq 0$ because damping acts only on $\mathbf{u}$,  where
$\mathcal{E}$ contains contributions from $a$ and $\b$ that
cannot be controlled by the dissipated quantity $\|\mathbf{u}\|^2$
without additional structure. In addition,
the nonlinear interactions $\b\cdot\nabla\b$,
$\b\cdot\nabla\mathbf{u}$, and $\mathbf{u}\cdot\nabla\b$
couple all three unknowns in a way that makes standard energy methods for
the viscous or resistive MHD systems
\cite{Dong-Wu-Zhai,L-X-Z,Li-Xu-Zhai,L-H,L-H1,Suen,SuenA,Wu-Xu-Zhai,Wu-Zhai,Wu-Zhu}
directly inapplicable. Cancellations that were exploited in those settings
are no longer accessible.

%% ------------------------------------------------------------------
\subsection{Stabilizing effect of the background magnetic field}
%% ------------------------------------------------------------------
It has long been recognized, from both physical experiments and numerical
simulations, that a sufficiently strong uniform background magnetic
	field can suppress turbulence and stabilize an electrically conducting
fluid.  This phenomenon was first identified theoretically by
Alfv\'{e}n \cite{ALF}, who showed that a uniform magnetic field supports
transverse wave propagation (now known as Alfv\'en waves), which
transport momentum and energy along field lines without dissipation.

\medskip
Experimental confirmation of this stabilizing effect is abundant.
Gallet, Berhanu, and Mordant \cite{GBM} reported measurements of forced
turbulence in a swirling liquid-metal flow (sodium) subject to an
externally imposed magnetic field. As the field strength was increased,
turbulent fluctuations were progressively suppressed and the flow
became quasi-two-dimensional and laminar along the field direction.
Califano and Chiuderi \cite{CC} demonstrated numerically that resistivity-independent dissipation of MHD waves occurs in inhomogeneous plasmas threaded by a background field, showing that even without resistivity the field geometry alone can transfer energy from large to small scales in a controlled manner.

\medskip
In engineering practice, the stabilizing effect of an external magnetic
field is exploited routinely in metallurgical continuous casting
processes \cite{Davidson}, where liquid steel is threaded by a DC
magnetic field to suppress turbulent fluctuations and control
solidification quality; in electromagnetic braking of liquid-metal
flows in nuclear blankets of fusion reactors \cite{Moreau}; and in
magnetohydrodynamic generators, where the imposition of a strong field
perpendicular to the flow is used to maintain laminar operation.
More recently, magnetic confinement in tokamak devices relies
fundamentally on the idea that a sufficiently strong and appropriately
structured magnetic field prevents the plasma from touching the wall,
a stability problem that is mathematically related to the one we study
here \cite{Cabannes-H,Landau-Lifshitz}.

\medskip
The mathematical mechanism underlying the stabilising effect is most
transparent at the level of the linearized system of \eqref{LinNon},
\begin{equation}\label{linearised}
	\left\{
	\begin{aligned}
		&\partial_t a + \mathrm{div}\,\mathbf{u} = 0, \\[3pt]
		&\partial_t\mathbf{u} + \beta\nabla a + \mathbf{u}
		= \omega\cdot\nabla\b - \nabla(\omega\cdot\b), \\[3pt]
		&\partial_t\b
		= \omega\cdot\nabla\mathbf{u} - \omega\,\mathrm{div}\,\mathbf{u}, \\[3pt]
		&\mathrm{div}\,\b = 0.
	\end{aligned}
	\right.
\end{equation}
In what follows, we intend  to explore the hidden wave structure  in \eqref{linearised} and to understand the
stability mechanism. To explain this clearly, taking the divergence of the momentum equation and using the continuity
equation yields a damped wave equation for the density perturbation,
\begin{equation}\label{wave-density}
	\partial_{tt} a - \beta\Delta a + \partial_t a
	=- \mathrm{div}\bigl(\omega\cdot\nabla\b - \nabla(\omega\cdot\b)\bigr).
\end{equation}
The left-hand side describes a damped acoustic wave with
sound speed $\sqrt{\beta}=\sqrt{P'(1)}$ and damping rate $1$.  The
right-hand side is a source driven by the magnetic field perturbation,
revealing the coupling between acoustic and magnetic modes.

\medskip
On the other hand, differentiating the induction equation \eqref{linearised}$_3$ in time
and substituting the momentum equation yields a wave equation for
	the magnetic field perturbation,
\begin{equation}\label{wave-magnetic}\begin{split}
	\partial_{tt}\b - (\omega\cdot\nabla)^2\b
	+ \partial_t\b
	&= -\omega\bigl[\partial_t\,\mathrm{div}\,\mathbf{u}+\mathrm{div}\,\mathbf{u}\bigr]
	- (\omega\cdot\nabla)\bigl(\beta\nabla a + \mathbf{u}\bigr)
	\\
&\quad+ (\omega\cdot\nabla)(\mathbf{u} - \nabla(\omega\cdot\b)).
\end{split}\end{equation}
The operator $(\omega\cdot\nabla)^2$ on the left-hand side governs
propagation of Alfv\'en waves along the background field $\omega$
at Alfv\'{e}n speed $|\omega|$.  The damping term $\partial_t\b$
arises indirectly through the velocity coupling and is the mathematical
embodiment of the physical stabilisation mechanism.

\medskip
From the above analysis, it is surprising that the coupled system \eqref{wave-density}--\eqref{wave-magnetic} for $a, \b$  reveals a double-wave structure:  acoustic waves in the density are driven
by magnetic fluctuations, while Alfv\'{e}nic waves in the magnetic field are
forced by acoustic and velocity perturbations.  It is this
two-way coupling, mediated by the background field $\omega$, that
disperses energy and prevents the concentration of high-frequency modes
that would otherwise lead to finite-time breakdown. The wave
structure  provides much more stabilization and regularization properties
than the original system \eqref{LinNon}  and hence, the stability result of
the solutions becomes possible.

\medskip
On the periodic torus $\mathbb{T}^3$, the wave operator $(\omega\cdot\nabla)$
acts on Fourier modes $e^{i\mathbf{k}\cdot x}$ as multiplication by
$i(\omega\cdot\mathbf{k})$.  For the resulting Poincar\'e-type inequality
to be useful, namely, for $(\omega\cdot\nabla)$ to be invertible with
controlled loss of derivatives, one needs the inner products
$\omega\cdot\mathbf{k}$ to be bounded away from zero for all nonzero
lattice vectors $\mathbf{k}$.  The precise condition is the
Diophantine condition introduced by Chen--Zhang--Zhou \cite{Chen-Zhang-Zhou}:
\begin{equation}\label{Diophantine0}
	\exists\; c>0 \;\text{ and }\; r>2 \;\text{ such that }\;
	|\omega\cdot\mathbf{k}| \geq \frac{c}{|\mathbf{k}|^r}, \qquad
	\forall\;\mathbf{k}\in\mathbb{Z}^3\setminus\{\mathbf{0}\}.
\end{equation}
By classical results in Diophantine approximation, condition
\eqref{Diophantine0} holds for Lebesgue-almost every $\omega\in\mathbb{R}^3$
\cite{Chen-Zhang-Zhou}. We remark that studying the dynamics near a vector field satisfying the
Diophantine condition has been a common practice in ergodic theory and
dynamical systems (see, e.g., \cite{Cas, Koc, Lop}). In recent years,  the Diophantine condition is frequently employed in the study of MHD equations such as  the well-posedness \cite{Jiang-Jang1,Wu-Zhai} and stability \cite{Liss,RZZ}. Here, the key consequence is the following Poincar\'{e}-type
inequality with derivative loss (see Lemma \ref{Diophantine-inequality}):
for any zero-mean function $f$ on $\mathbb{T}^3$,
\begin{equation}\label{Poincare-Diophantine}
	\|f\|_{H^s(\mathbb{T}^3)} \leq C\,\|\omega\cdot\nabla f\|_{H^{s+r}(\mathbb{T}^3)}.
\end{equation}
This inequality is the main tool that converts the ``hidden'' dissipation
generated by the wave structure into quantitative Sobolev control of
$a$ and $\b$.

%% ------------------------------------------------------------------
\subsection{Related mathematical results}
%% ------------------------------------------------------------------
It is well known that,
the global well-posedness of the  following standard compressible MHD system
\eqref{MHD-begin} with full viscosity ($\mu>0, \mu+\lambda>0$) and resistivity
($\nu>0$)
\begin{eqnarray}\label{MHD-begin}
	\left\{\begin{aligned}
		&\partial_t \rho + {\mathrm{div}} (\rho \mathbf{u}) =0,  \\
		&\partial_{t}(\rho \u)+\Dv (\rho \u\otimes \u)-\mu\Delta \mathbf{u} - (\lambda+\mu) \nabla {\mathrm{div}} \mathbf{u}+\nabla P=\H\cdot \nabla \H-\frac{1}{2}\nabla|\mathbf{H}|^2,\\
		&\partial_t \mathbf{H}+\mathbf{u}\cdot\nabla\mathbf{H}-\mathbf{H}\cdot\nabla\mathbf{u}-\nu\Delta \mathbf{H}
		=-\mathbf{H}\div\mathbf{u},\\
		&{\mathrm{div}} \mathbf{H} =0,\\
		&(\rho,\u,\H)_{|t=0}=(\rho_0,\u_0,\H_0),
	\end{aligned}\right.
\end{eqnarray}
is by now well understood (see, e.g.,
\cite{DF,Gao-Wu-Xu,Hu-Wang1,Hu-Wang2,Kawashima,KawashimaS,K-L,L-X-Z,L-M-X,L-H,L-H1,Suen,SuenA,VH} and the references therein).

\medskip
When only one of viscosity or resistivity is present, the problem becomes
substantially harder.  For the viscous and non-resistive compressible MHD
($\mu>0, \mu+\lambda>0$,$\nu=0$), the global solutions in $\mathbb{R}^3$ for small
data remain open; on the torus, significant progress has been made in
\cite{Dong-Wu-Zhai,Wu-Wu,Wu-Zhai,Wu-Zhu,Li-Xu-Zhai,Jiu-Liu-Xie}, where
the background magnetic field and the Diophantine condition are
exploited to generate effective dissipation for the magnetic field.

\medskip
The stability of MHD equilibria near a background magnetic field has
attracted substantial attention in the incompressible setting. For the incompressible viscous and non-resistive MHD system
($\mu>0$, $\nu=0$, $\mathrm{div}\,\mathbf{u}=0$), Lin--Zhang
\cite{Lin-Zhang} and Lin--Xu--Zhang \cite{Lin-Zhang1} initiated the
systematic stability study in two and three space dimensions
respectively, proving global well-posedness for small perturbations of
a uniform background field.  Their work inspired a large body of further
investigations covering various partial dissipation settings (see, e.g.,
\cite{Chen-Zhang-Zhou, GBM, L-Zhang, Jiang-Jang1, PZZ,
	Ren-Wu-Xia-Zhang, Ren-Xiang-Zhang, Tan-Wang, Wu-Zhu1, Xie-Jiu-Liu, Xu-Zhang,
	 zhaijde, Zhou-Zhu} and the references therein.
For the incompressible ideal MHD ($\mu=0$, $\nu=0$, $\mathrm{div}\,\mathbf{u}=0$),
Bardos--Sulem--Sulem \cite{BSS} proved global existence in H\"older spaces
near a strong background field using the Els\"asser-variable approach,
and Cai \cite{Cai} recently extended this to the inhomogeneous
incompressible case.  In the incompressible setting with velocity damping,
Wu--Wu--Xu \cite{Wu-Wu-Xu} established global well-posedness for the 2D
system, Jiang--Jiang--Zhao \cite{JJZ} treated the 3D
initial-boundary-value problem in a horizontally periodic strip. Jiang--Jiang--Zhao \cite{JJZ1} studied the Rayleigh-Taylor instability  to the 2D inhomogeneous, incompressible, inviscid case with velocity damping  in a  horizontal slab domain.

\medskip
For the compressible inviscid resistive  MHD ($\mu=\lambda=0$),
Wu--Xu--Zhai \cite{Wu-Xu-Zhai} established global smooth solutions to the
non-isentropic equations on $\mathbb{T}^3$, and Li--Qiao \cite{Li-Qiao}
obtained the analogous result for the inviscid resistive isentropic
case.  Neither result covers the fully ideal ($\mu=\lambda=\nu=0$) isentropic
setting, which is the subject of the present paper.

%% ------------------------------------------------------------------
\subsection{Main result}
%% ------------------------------------------------------------------

We can now state the main theorem of this paper.

\begin{theorem}\label{dingli}
	Let $N>3r+3$ with $r>2$, and assume $\omega\in\mathbb{R}^3$ satisfy
	the Diophantine condition \eqref{Diophantine0} for some constant $c>0$.
	Suppose the initial data $(\rho_0-1, \mathbf{u}_0, \b_0)\in H^N(\mathbb{T}^3)$
	satisfies the mean conditions
	\begin{equation}\label{junzhi0}
		\int_{\mathbb{T}^3}\rho_0\,dx = 1, \qquad
		\int_{\mathbb{T}^3}\b_0\,dx = 0,
	\end{equation}
	and $\frac12\leq\rho_0\leq \frac32.$  Then there exists $\varepsilon>0$ such that if
	\begin{equation}\label{2026-6-27-3-1}
	\|\rho_0-1\|_{H^N} + \|\mathbf{u}_0\|_{H^N}
	+ \|\b_0\|_{H^N} \leq \varepsilon,
	\end{equation}
	the system \eqref{MHD1} admits a unique global solution
	$(\rho-1, \mathbf{u}, \b)\in C([0,\infty); H^N)$.
	Moreover, for any $t\geq 0$ and $r+1\leq\gamma< N$,
	\begin{equation}\label{decay}
		\|\rho(t)-1\|_{H^{\gamma}}
		+ \|\mathbf{u}(t)\|_{H^\gamma}
		+ \|\b(t)\|_{H^\gamma}
		\leq C(1+t)^{-\frac{N-\gamma}{2r+2}}.
	\end{equation}
\end{theorem}

\begin{remark}
	To the best of our knowledge, Theorem \ref{dingli} is the first
	global well-posedness result for the triple dimensional isentropic
	compressible inviscid, non-resistive MHD equations \eqref{MHD1111}.
	When $\b_0\equiv 0$, the system reduces to the compressible Euler
	equations with velocity damping \cite{STW,WY}, for which global smooth
	solutions near a constant state are known.  Our result thus demonstrates
	that the addition of a non-zero background magnetic field $\omega$ is
	compatible with and in fact enhances the stabilizing effect of damping,
	bringing the ideal MHD system to the same level of understanding as the
	compressible Euler equations.
\end{remark}

\begin{remark}
	Compared with \cite{Li-Qiao,Wang-Xin,Wu-Xu-Zhai}, the present paper is
	the first to treat the fully ideal (inviscid and non-resistive) isentropic
	compressible MHD system.
\end{remark}

\begin{remark}\label{remar3}
	The mean conditions \eqref{junzhi0} are preserved under the time
	evolution of \eqref{MHD1}: for all $t\geq 0$,
	\begin{equation}\label{junzhi1}
	\int_{\mathbb{T}^3}\rho\,dx = 1, \qquad
	\int_{\mathbb{T}^3}\b\,dx = 0.
	\end{equation}
	Indeed, integrating the density equation over $\mathbb{T}^3$ gives
	$\frac{d}{dt}\int\rho\,dx = 0$; integrating the induction equation
	and using $\mathrm{div}\,\b=0$ and $\omega\cdot\nabla(\cdot)$
	having zero mean on $\mathbb{T}^3$ gives $\frac{d}{dt}\int\b\,dx=0$.
	%and integrating the momentum equation gives $\frac{d}{dt}\int\rho\mathbf{u}\,dx = -\int\rho\mathbf{u}\,dx$, so the mean of $\rho\mathbf{u}$ decays exponentially from zero.
\end{remark}

%% ------------------------------------------------------------------
\subsection{Strategy of the proof}
%% ------------------------------------------------------------------

The proof of Theorem \ref{dingli} is based on a bootstrap argument in
which one propagates a global smallness bound on
$\sup_{t\geq 0}\|(a,\mathbf{u},\b)\|_{H^N}$ by deriving
time-uniform a priori estimates.  The main steps are as follows.

\medskip
The first step is to exploit hidden dissipation from the background
	field (Lemmas \ref{lemma-pri3-3-1}--\ref{lemma-pri3-3}). In fact,
the absence of diffusion in the density and magnetic field equations
means that the following  energy inequality (see Lemma \ref{lemma-pri1})
%\begin{equation}\label{basic-energy}\frac{d}{dt}\int\Bigl(\rho|\mathbf{u}|^2 + |\b|^2+ e(\rho)\Bigr)\,dx + \int\rho|\mathbf{u}|^2\,dx = 0\end{equation}(where $e(\rho)$ is the potential energy density)
\begin{equation}\begin{split}\label{2026-6-27-pri-2-2}
\frac{d}{dt}E_s+
\|(u,\Lambda^s\u)\|_{L^2}^2&\leq C(1+\|(a,\u,\b)\|_{H^3})\|(a,\u,\b)\|_{H^3}\|(a,\u,\b)\|^2_{H^s}\\&\quad+
\|a\|_{H^3}\|\b\|^{2}_{H^3}\|a\|_{H^s}\|\u\|_{H^s}
\end{split}\end{equation}
(where $E_s\approx\|(a,\u,\b)\|_{H^s}^2$) provides control only
on $\u$ and not on $a$ or $\b$. In order to close this energy, the greatest difficulty   is to achieve dissipation
of density and magnetic fields. Here,  the non-zero background magnetic field $\omega$ satisfying the Diophantine condition plays a key role.
On the one hand, to reveal hidden dissipation for the density $a$: we project the
momentum equation onto $\omega$ to isolate the term $\omega\cdot\nabla a$;
applying $\Lambda^s$ and testing against $\omega\cdot\nabla\Lambda^s a$
then yields
\begin{equation}\label{density-dissipation-sketch}
	\beta\|(\omega\cdot\nabla)\Lambda^s a\|_{L^2}^2
	+ \frac{d}{dt}\int\omega\cdot\nabla\Lambda^s a\cdot\Lambda^s(\mathbf{u}\cdot\omega)\,dx
	\lesssim \|\mathbf{u}\|_{H^{s+1}}^2 + \|(f_1,f_2)\|_{H^s}^2.
\end{equation}
The Poincar\'{e} inequality \eqref{Poincare-Diophantine} then converts
$\|(\omega\cdot\nabla)\Lambda^s a\|_{L^2}$ into $\|a\|_{H^{s-r}}$, at the cost
of $r$ derivatives. On the other hand,
to discover hidden dissipation for the magnetic field $\b$:
applying $\Lambda^s$ to the momentum equation and testing against
$(\omega\cdot\nabla)\Lambda^s\b$ yields
\begin{equation}\label{magnetic-dissipation-sketch}
	\|(\omega\cdot\nabla)\Lambda^s\b\|_{L^2}^2
	- \frac{d}{dt}\int\Lambda^s\mathbf{u}\cdot(\omega\cdot\nabla)\Lambda^s\b\,dx
	\lesssim \|\mathbf{u}\|_{H^{s+1}}^2+\|(f_2,f_3)\|_{H^s}^2.
\end{equation}
Again, \eqref{Poincare-Diophantine} converts this into $H^{s-r}$ control
of $\b$.  At this point, together with the dissipation of $u$ itself, this yields effective dissipation for $(a,u,\b)$ at the $H^{s-r}$  level.  %Summing over $s$ and iterating in the derivative index(Lemma \ref{lemma-pri2}) produces effective dissipation at the level of$H^{2r}$ for all three unknowns.

\medskip
The second step is to bound the nonlinear error terms.
The nonlinear terms $f_1, f_2, f_3$ in \eqref{nonlinear-terms} must be
estimated in high-order Sobolev norms.  Because no diffusion is available
for $a$ or $\b$, the standard techniques for viscous and/or
resistive MHD
\cite{Dong-Wu-Zhai,L-X-Z,Li-Xu-Zhai,L-H,L-H1,Suen,SuenA,Wu-Xu-Zhai,Wu-Zhai,Wu-Zhu}
are not directly applicable.  Instead, we use commutator estimates
(Kato--Ponce type, Lemma \ref{Hk-estimate-7-8}), integration by parts, and exploit
structural cancellations  by rearranging some nonlinear terms. For example, we here reconsider some combinations of the nonlinear terms  such as $\big|\int\Lambda^s(\b\cdot\nabla \b)\cdot\Lambda^s \u dx+\int\Lambda^s(\b\cdot\nabla \u)\cdot\Lambda^s\b dx\big|$ in the pair $(\mathrm I_5,\mathrm I_8)$, $\big|-\frac12\int\Lambda^s\nabla|\b|^2\cdot\Lambda^s\u dx-\int\Lambda^s(\b\textrm{div}\u)\cdot\Lambda^s\b dx\big|$ in the pair $(\mathrm I_6,\mathrm I_9)$ stated in Lemma \ref{lemma-pri1},  and  $\int(\omega\cdot\nabla)\Lambda^{s}\b\cdot\nabla\Lambda^{s}(\beta a+ \b\cdot\omega)\,dx=0$ stated in Lemma \ref{lemma-pri3-3}  that arise from the divergence-free constraint $\mathrm{div}\,\b=0$.

\medskip
The third step is to close the energy hierarchy. Our proof here mainly relies on a two-tier
energy method. Combining the  lower-order energy inequality \eqref{2026-6-27-pri-2-2} with $s=r+1$ (which provides dissipation for $\u$) with the hidden dissipation estimates from the first step (at order $r$),
%Combining the basic energy estimate \eqref{basic-energy} (which provides damping for $\mathbf{u}$) with the hidden dissipation estimates from the first step(at order $4r+2$) and the high-order energy inequality (at order $N$),
we can  define the  following Lyapunov functional
\[
{\mathcal{E}(t)}=AE_{r+1}(t)+\int \Big(\omega\cdot\nabla\Lambda^{r}a\cdot\Lambda^{r}(\u\cdot\omega)-\Lambda^{r}\u\cdot(\omega\cdot\nabla)\Lambda^{r}\b\Big) \,dx,
\]
for a sufficiently large constant $A>0$.  The cross terms are bounded by
Cauchy--Schwarz, so that
$\mathcal{E}(t)\approx\|(a,\mathbf{u},\b)\|_{H^{r+1}}^2$.
Under the bootstrap assumption
$\sup_{t\in[0,T]}\|(a,\mathbf{u},\b)\|_{H^N}\leq\delta$, \eqref{Poincare-Diophantine} and the interpolation inequality, one
derives
\[
\frac{d}{dt}\mathcal{E}(t) + c\|(a,\mathbf{u},\b)\|_{0}^2 \leq 0.
\]
An interpolation inequality then gives $\mathcal{E}(t)\lesssim (1+t)^{-(N-r-1)/(r+1)}$.
This decay of the lower-order energy is fed back into the highest-order
energy inequality \eqref{2026-6-27-pri-2-2} with $s=N$ to show via Gronwall's inequality that
$\|(a,\mathbf{u},\b)\|_{H^N}^2\leq C\varepsilon^2$, closing the
bootstrap when $\varepsilon$ is small enough.  The decay rate
\eqref{decay} follows by  the interpolation between the lower and
highest-order estimates.

\medskip
\noindent\textbf{Notations.}
We let $C$ denote a generic positive constant that may vary from line to
line and write $A\lesssim B$ to mean $A\leq CB$.  The spaces
$L^q(\mathbb{T}^3)$ and $H^k(\mathbb{T}^3)$ carry their standard norms
$\|\cdot\|_{L^q}$ and $\|\cdot\|_k$, with $H^0=L^2$.  We write
$[U,V]W = U(VW)-V(UW)$ for the commutator, and use the shorthand
$\int f\,dx = \int_{\mathbb{T}^3}f\,dx$,
$\|(f,g)\|_X^2 = \|f\|_X^2+\|g\|_X^2$,
$\|h(f,g)\|_X^2 = \|(hf,hg)\|_X^2$.

\medskip
The rest of this paper is organized as follows. Section \ref{sec2} collects the preliminary lemmas, chiefly the
Poincar\'e-type inequality from the Diophantine condition.
Section \ref{sec3} derives the time-uniform \emph{a priori} estimates.
Section \ref{sec:pro} completes the proof of Theorem \ref{dingli}.

\vskip .3in
\section{Preliminaries}
\label{sec2}

This section collects the several analytical tools used throughout the paper.
The first is a Poincar\'{e}-type inequality derived from the
Diophantine condition \eqref{Diophantine0}.

\begin{lemma}[Diophantine--Poincar\'{e} inequality, {\cite{Xie-Jiu-Liu}}]
	\label{Diophantine-inequality}
	Assume that $\omega\in\mathbb{R}^3$ satisfies the Diophantine condition
	\eqref{Diophantine0} with constants $c>0$ and $r>2$.  Let $\geq0$.
	Then for any function $f$ with $\nabla f\in H^{s+r}(\mathbb{T}^3)$ and
	$\int_{\mathbb{T}^3}f\,dx = 0$, one has
	\begin{equation}\label{Poincare-Dioph}
		\|f\|_{H^s(\mathbb{T}^3)} \leq C\,\|\omega\cdot\nabla\Lambda^{s+r} f\|_{L^2(\mathbb{T}^3)},
	\end{equation}
	where the constant $C$ depends only on $c$, $r$, $|\omega|$, and
	$\mathbb{T}^3$.
\end{lemma}

The inequality \eqref{Poincare-Dioph} should be understood as a
Poincar\'{e} inequality with derivative loss. It replaces the
	standard bound $\|f\|_{H^s}\lesssim\|\nabla f\|_{H^s}$ (which would
	require $\omega$ to be nonvanishing at every frequency) with a bound
	involving $r$ additional derivatives on the right-hand side, valid for
	almost every $\omega\in\mathbb{R}^3$.  The loss of $r$ derivatives is
	precisely what forces the high-regularity requirement $N>3r+3$ in
	Theorem~\ref{dingli}.

\medskip
The following three lemmas provide the tools needed to estimate the
nonlinear terms.  In each case the periodic torus $\mathbb{T}^3$ replaces
$\mathbb{R}^3$, and the estimates hold with the same constants as their
classical counterparts.

\begin{lemma}[Gagliardo--Nirenberg inequality on $\mathbb{T}^3$,
	{\cite{Nirenberg-1959}}]
	\label{lem:GN}
	Let $0\leq m, s\leq\kappa$.  Then
	\[
	\|\Lambda^s f\|_{L^p}
	\leq C\,\|\Lambda^m f\|_{L^q}^{1-\theta}\,\|\Lambda^\kappa f\|_{L^r}^\theta,
	\]
	where $\theta\in[0,1]$ satisfies
	$\frac{s}{3}-\frac{1}{p}=\bigl(\frac{m}{3}-\frac{1}{q}\bigr)(1-\theta)
	+\bigl(\frac{\kappa}{3}-\frac{1}{r}\bigr)\theta$.
	When $p=\infty$, one requires $\theta\in (0,1)$.
\end{lemma}

\begin{lemma}[Kato--Ponce product estimate, {\cite{kato}}]
	\label{Hk-estimate}
	Let $s\geq 0$.  For any
	$f,g\in H^s(\mathbb{T}^3)\cap L^\infty(\mathbb{T}^3)$,
	\begin{equation}\label{KP-product}
		\|fg\|_{H^s}
		\leq C\bigl(\|f\|_{L^\infty}\|g\|_{H^s}
		+ \|g\|_{L^\infty}\|f\|_{H^s}\bigr).
	\end{equation}
\end{lemma}

\begin{lemma}[Kato--Ponce commutator estimate, {\cite{kato}}]
	\label{Hk-estimate-7-8}
	Let $s>0$.  For any
	$f\in H^s(\mathbb{T}^3)\cap W^{1,\infty}(\mathbb{T}^3)$ and
	$g\in H^{s-1}(\mathbb{T}^3)\cap L^\infty(\mathbb{T}^3)$,
	\begin{equation}\label{KP-commutator}
		\|[\Lambda^s, f\cdot\nabla]g\|_{L^2}
		\leq C\bigl(\|\nabla f\|_{L^\infty}\|\Lambda^s g\|_{L^2}
		+ \|\Lambda^s f\|_{L^2}\|\nabla g\|_{L^\infty}\bigr).
	\end{equation}
\end{lemma}

\begin{lemma}[Sobolev estimate for compositions, {\cite{Triebel}}]
	\label{Hk-estimate-1}
	Let $s>0$, $f\in H^s(\mathbb{T}^3)\cap L^\infty(\mathbb{T}^3)$, and
	let $F\in C^\infty(\mathbb{R})$ with $F(0)=0$.  Then
	\begin{equation}\label{composition}
		\|F(f)\|_{H^s}
		\leq C\bigl(1+\|f\|_{L^\infty}\bigr)^{[s]+1}\|f\|_{H^s},
	\end{equation}
	where the constant $C$ depends on
	$\sup_{k\leq[s]+2,\,|t|\leq\|f\|_{L^\infty}}|F^{(k)}(t)|$.
\end{lemma}

	Lemma~\ref{Hk-estimate-1} is applied
	to handle the nonlinear pressure term $P'(\rho)-P'(1)$ and the factor
	$\rho^{-1}$, both of which are smooth functions of $a=\rho-1$ vanishing
	at $a=0$.  The uniform bound $\rho\in[\frac{1}{2},\frac32]$
	ensures the $L^\infty$ hypotheses.

%\medskip
%The next lemma controls $\|\mathbf{u}\|_{L^2}$ using the conservation law $\int_{\mathbb{T}^3}\rho\mathbf{u}\,dx=0$.  This is necessary because the system \eqref{MHD1} conserves the weighted mean $\int\rho\mathbf{u}\,dx$ rather than the unweighted mean $\int\mathbf{u}\,dx$.

%\begin{lemma}[Weighted Poincar\'{e} inequality,{\cite[Lemma~5.1]{Danchin-Mucha}}]\label{Poincare-inequality}
%Let $f\in L^2(\Omega)$ be a nonneg\-ative, nonzero measurable function on an open bounded domain $\Omega$ with $\partial\Omega\in C^1$.  There exists a constant $K>0$ depending only on $\Omega$ such that for any $z\in H^1(\Omega)$,
	%\begin{equation}\label{weighted-Poincare}\|z\|_{L^2}\leq \frac{1}{M}\Bigl|\int_\Omega fz\,dx\Bigr|+ K\Bigl(1+\frac{1}{M}\|M-f\|_{L^2}\Bigr)\|\nabla z\|_{L^2},\qquad M := \int_\Omega f\,dx.\end{equation}\end{lemma}

	%Applying \eqref{weighted-Poincare} with $f=\rho$, $z=u_i$ (the $i$-th component of $\mathbf{u}$), $M=\bar\rho=1$, and using $\int_{\mathbb{T}^3}\rho\mathbf{u}\,dx=0$, one obtains
	%\begin{equation}\label{u-Poincare}\|\mathbf{u}\|_{L^2}\leq K\bigl(1+\|a\|_{L^2}\bigr)\|\nabla\mathbf{u}\|_{L^2}.\end{equation}
	%Under the smallness assumption $\|a\|_{H^N}\leq\delta\ll 1$, the factor $(1+\|a\|_{L^2})$ is bounded, and \eqref{u-Poincare} reduces to the usual Poincar\'{e} inequality up to a controlled multiplicative constant. This estimate is used throughout Section~\ref{sec3} to pass between $\|\Lambda^s\mathbf{u}\|_{L^2}$ and $\|\mathbf{u}\|_{H^s}$.

\section{ A priori estimates}
\label{sec3}
This section is devoted to deriving \emph{a priori} estimates for smooth solutions to system \eqref{MHD1}. To this end, throughout this section, we assume that $(\rho,\u,\b)\in C([0,T];H^{N}(\mathbb T^3))$ is a smooth solution to system \eqref{MHD1} on $\mathbb T^3\times[0,T]$ for some $T>0$. Moreover,
\begin{equation}\begin{split}\label{pri-2}
		\frac{1}{2}\le \rho(t,x) \le \frac32, \quad \text{for all } (x,t)\in\mathbb T^3\times[0,T].
\end{split}\end{equation}
For simplicity, we set
$$
Q:=\sup_{0\leq k\leq
N+6,s\in(\frac12,\frac32)}|p^{(k)}(s)|,\quad\beta:=p'(1),\quad \Lambda:=\sqrt{-\Delta}.
$$
Denote by $e(\rho)$ the potential energy density, namely
$$
e(\rho)=2\rho\int_{1}^\rho\frac{p(s)-p(1)}{s^2}ds.
$$
Recalling that $p(\rho)$ is increasing in $\rho$, it is clear that $e(\rho)>0$ for $\rho\in(0,1)\cup(1,\infty)$ and $e(1)=0$.
Taking $L^2$ inner product of $\eqref{MHD1}_{2,3}$ with $(u,h)$ and using $\eqref{MHD1}_1$, one gets by direct calculations that
\begin{equation}\begin{split}\label{pri-1}
\frac12\frac{d}{dt}\int\Big(\rho|\u|^2+|\b|^2+e(\rho)\Big)\,dx+\int\rho|\u|^2\,dx=0.
\end{split}\end{equation}
%Because of the embedding relation and  the  smallness of the norm in $H^N(\T^3)$,
We easily obtain by \eqref{pri-2} that
\begin{equation*}%\label{eq:smallad}
\sup_{t\in\R_+,\, x\in\T^3} |a(t,x)|\leq \frac12.
\end{equation*}
Noting $e(1)=e'(1)=0$ and $e''(\rho)=\frac{2p'(\rho)}{\rho}>0$ and $a\in(-\frac12,\frac12)$,
one obtains by  Taylor's expansion that
 \begin{equation*}\begin{split}
\Big(\inf_{\rho\in(\frac12,\frac32)}\frac{p'(\rho)}{\rho}\Big)a^2\leq e(a+1)\leq \Big(\sup_{\rho\in(\frac12,\frac32)}\frac{p'(\rho)}{\rho}\Big)a^2,
\end{split}\end{equation*}
which  together with \eqref{pri-2} and \eqref{pri-1} implies that
\begin{equation}\begin{split}\label{pri-1-1}
\frac{d}{dt}\int\Big(\rho|\u|^2+|\b|^2+e(\rho)\Big)\,dx+\|\u\|_0^2\leq0
\end{split}\end{equation}
with \begin{equation}\begin{split}\label{2026-6-27-pri-1-1}&\min\{\inf_{\rho\in(\frac12,\frac32)}\frac{p'(\rho)}{\rho},\frac12\}\|(a,\u,\b)\|^2_0\\&\leq\int\Big(\rho|\u|^2+|\b|^2+e(\rho)\Big)\,dx\\
&\leq\max\{\sup_{\rho\in(\frac12,\frac32)}\frac{p'(\rho)}{\rho},\frac32\}\|(a,\u,\b)\|^2_0.\end{split}\end{equation}
%Recalling $\int \b \,dx=0$ in  (\ref{junzhi1}), Poincar\'{e}'s inequality , for an absolute positive constant $C$, gives rise to,
%\begin{equation}\begin{split}\label{pri-3}\|\b\|_0\leq C\|\nabla \b\|_0.\end{split}\end{equation}
%Due to $\int\rho \u \,dx=0$ and $\int\rho \,dx=1$ in (\ref{junzhi1}), it follows from Lemma \ref{Poincare-inequality} that
%\begin{equation}\begin{split}\label{pri-7-22-1}\|\u\|_0\leq\big|\int \rho \u \,dx\big|+K\big(1+\|\rho-1\|_{0}\big)\|\nabla \u\|_{0}\leq K(1+\|a\|_{0})\|\nabla \u\|_0.\end{split}\end{equation}
%Thanks to (\ref{pri-3}) and (\ref{pri-7-22-1}), one gets by  Poincar\'{e}'s inequality that
%\begin{equation}\label{EQUIVN}\begin{split}
  %\|\Lambda^sa\|_0\leq\|a\|_s\leq C_s\|\Lambda^sa\|_0,\quad\|\Lambda^s\b\|_0\leq\|\b\|_s\leq C_s\|\Lambda^s\b\|_0,\quad\|\Lambda^s\u\|_0\leq\|\u\|_s\leq C_s(1+\|a\|_{0})\|\Lambda^s\u\|_0,\end{split}\end{equation}
%for any nonnegative integer $s$, where $C_s\geq1$ is a constant depending only on $s$.This fact will be used throughout this section without further mentions.

\begin{lemma}\label{lemma-pri1}
Let $(\rho,\u,\b) \in C([0, T];H^N)$ be a solution to the  system \eqref{MHD1}. Denote
$$E_s(t):=\int\Big(\rho|\u|^2+|\b|^2+e(\rho)\Big)\,dx+\Big\|\big(\frac{p'}{\rho}\Lambda^sa,\sqrt{\rho}\Lambda^s\u,\Lambda^s\b\big)\Big\|^2_0.$$ Then, for any $3\leq s\leq N$, it holds that
\begin{align}\label{pri-2-2}
\frac{d}{dt}E_s+
\|(u,\Lambda^s\u)\|_0^2\leq C(1+\|(a,\u,\b)\|_3)\|(a,\u,\b)\|_3\|(a,\u,\b)\|^2_s+
\|a\|_3\|\b\|^{2}_3\|a\|_s\|\u\|_s
\end{align}
with
\begin{equation}
\begin{split}\label{bu-2026-6-27-1}& E_s\geq\min\{\inf_{\rho\in(\frac12,\frac32)}\frac{p'(\rho)}{\rho},\frac12\}\big(\|(a,\u,\b)\|^2_0+\|\Lambda^s(a,\u,\b)\|^2_0\big),\\&E_s\leq\max\{\sup_{\rho\in(\frac12,\frac32)}\frac{p'(\rho)}{\rho},\frac32\}\big(\|(a,\u,\b)\|^2_0+\|\Lambda^s(a,\u,\b)\|^2_0\big),\end{split}\end{equation}
where the positive constant $C$ depends only on $|\omega|,\beta,$ and $Q$.
\end{lemma}

\begin{proof} %Obviously, \eqref{pri-2-2} with $\ell=0$ is the basic energy inequality in \eqref{pri-1-1}. We consider the case when $\ell\ge 1$.
Applying $\Lambda^s$ %with $1\leq s\leq \ell$
to $\eqref{MHD1}_{2}$, $\eqref{MHD1}_{3}$ and taking $L^2$ inner product with $\Lambda^s\u$ and $\Lambda^s \b,$ respectively,  then integrating by parts and using $\eqref{MHD1}_{1}$, $\eqref{MHD1}_{4}$ yield
\begin{equation}\begin{split}\label{pri-4}
&\frac12\frac{d}{dt}\|(\sqrt\rho\Lambda^s\u,\Lambda^s\b)\|^2_0+\|\rho\Lambda^{s}\u\|_0^2
%&\quad=\int\Big[\Big(-[\Lambda^s,\rho]\u_t-[\Lambda^s,\rho \u\cdot\nabla]\u-p'\nabla\Lambda^s\rho-[\Lambda^s,p']\nabla\rho\\
%&\qquad-\Lambda^s(\rho \u)+\Lambda^s(\b\cdot\nabla \b)-\frac12\Lambda^s\nabla(|\b|^2)\Big)\cdot\Lambda^s \u\\
%&\qquad-\Big(\Lambda^s(\u\cdot\nabla \b)-\Lambda^s(\b\cdot\nabla \u)+\Lambda^s(\b\div \u)\Big)\cdot\Lambda^s \b\Big]\,dx\\
=:\sum_{j=1}^{10}\textrm{I}_j,
\end{split}\end{equation}
where
\begin{align*}
  &\textrm{I}_1:=-\int[\Lambda^s,\rho]\u_t\cdot \Lambda^s\u \,dx,&&\textrm{I}_2:=-\int[\Lambda^s,\rho \u\cdot\nabla]\u\cdot \Lambda^s\u \,dx,\\
   &\textrm{I}_3:=-\int p'\nabla\Lambda^s\rho\cdot \Lambda^s\u \,dx,&&\textrm{I}_4:=-\int [\Lambda^s,p']\nabla\rho\cdot\Lambda^s \u \,dx,\\
   &\textrm{I}_5:=\int \Lambda^s(\b\cdot\nabla \b)\cdot\Lambda^s \u \,dx,&&\textrm{I}_6:= -\frac12\int \Lambda^s\nabla(|\b|^2)\cdot\Lambda^s \u \,dx,\\
   &\textrm{I}_7:=-\int\Lambda^s(\u\cdot\nabla \b)\cdot\Lambda^s\b \,dx,&&\textrm{I}_8:=\int\Lambda^s(\b\cdot\nabla \u)\cdot\Lambda^s\b \,dx,\\ &\textrm{I}_9:=-\int\Lambda^s(\b\div \u)\cdot\Lambda^s \b \,dx,&&\textrm{I}_{10}:=-\int[\Lambda^s,\rho]\u\cdot\Lambda^s \u \,dx.
\end{align*}
In what follows, we shall bound each term above.  We first estimate $\textrm{I}_3$.
Using $\eqref{MHD1}_{1}$, one has
\begin{align*}
  \rho\Lambda^s\text{div}\u=&\Lambda^s(\rho\text{div}\u)-[\Lambda^s,\rho]\text{div}\u\\
  %=-\Lambda^s(\partial_t\rho+\u\cdot\nabla\rho)-[\Lambda^s,\rho]
 % \text{div}\u\\
  =&-(\partial_t\Lambda^s\rho+\u\cdot\nabla\Lambda^s\rho)-[\Lambda^s,\u\cdot\nabla]\rho-[\Lambda^s,\rho]\text{div}\u.
\end{align*}
Noticing that
$$
\Big(\frac{p'}{\rho}\Big)_t+\div \Big(\frac{p'}{\rho}\u\Big)=\Big(\frac{2p'}{\rho}-p''\Big)\div \u,
$$
it follows from the integration by parts formula that
\begin{equation*}\begin{split}
\textrm{I}_3&=\int\Lambda^s\rho\big(p'\div \Lambda^s\u+\nabla p'\cdot \Lambda^s\u\big)\,dx\\
%&=\int\Big[-\frac{p'}{\rho}\Lambda^s\rho\Big(\partial_t\Lambda^s\rho+\u\cdot\nabla\Lambda^s\rho
%+[\Lambda^s,\u\cdot\nabla]\rho+[\Lambda^s,\rho]\div \u\Big)+\Lambda^s\rho\nabla p'\cdot\Lambda^s\u\Big]\,dx\\
%&=-\frac12\frac{d}{dt}\int\frac{p'}{\rho}|\Lambda^s\rho|^2\,dx+\frac12\int\Big[\big(\frac{p'}{\rho}\big)_t+\div
%\big(\frac{p'}{\rho}\u\big)\Big]|\Lambda^s\rho|^2\,dx-\int\frac{p'}{\rho}\Lambda^s\rho[\Lambda^s,\u\cdot\nabla]\rho \,dx\\&\quad-\int\frac{p'}{\rho}\Lambda^s\rho[\Lambda^s,\rho]\div \u \,dx+\int\Lambda^s\rho\nabla p'\cdot\Lambda^s\u \,dx\\
&=-\frac12\frac{d}{dt}\int\frac{p'}{\rho}|\Lambda^s\rho|^2\,dx+\frac12\int\big(\frac{2p'}{\rho}-p''\big)\div \u|\Lambda^s\rho|^2\,dx
-\int\frac{p'}{\rho}\Lambda^s\rho[\Lambda^s,\u\cdot\nabla]\rho \,dx
\\&\quad-\int\frac{p'}{\rho}\Lambda^s\rho[\Lambda^s,\rho]\div \u \,dx+\int\Lambda^s\rho\nabla p'\cdot\Lambda^s\u \,dx.
\end{split}\end{equation*}
Since $\rho\in(\frac12,\frac32)$ and $p(\rho)$ is a smooth function in $\rho$, we conclude  from Lemma \ref{Hk-estimate-7-8}, the embedding relation , that
\begin{equation*}\begin{split}
\Big|\frac12\int\Big(\frac{2p'}{\rho}-p''\Big)\div \u|\Lambda^s\rho|^2\,dx\Big|
&\lesssim\Big\|\frac{2p'}{\rho}-p''\Big\|_{L^\infty}\|\nabla\u\|_{L^\infty}\|a\|^2_s
%\lesssim\|\nabla\u\|_{L^\infty}\|a\|^2_s
\lesssim\|\u\|_{3}\|a\|^2_s,\\
\Big|-\int\frac{p'}{\rho}\Lambda^s\rho[\Lambda^s,\u\cdot\nabla]\rho \,dx\Big|&\lesssim\Big\|\frac{p'}{\rho}\Big\|_{L^\infty}\|a\|_s\|[\Lambda^s,\u\cdot\nabla]\rho\|_0
\lesssim\|(a,\u)\|_{3}\|(a,\u)\|^2_s,\\
\Big|-\int\frac{p'}{\rho}\Lambda^s\rho[\Lambda^s,\rho]\div \u \,dx\Big|
&\lesssim\Big\|\frac{p'}{\rho}\Big\|_{L^\infty}\|a\|_s\|[\Lambda^s,\rho]\div \u\|_0
\lesssim\|(a,\u)\|_{3}\|(a,\u)\|^2_s,
\end{split}\end{equation*}
and
\begin{equation*}\begin{split}
\Big|\int\Lambda^s\rho\nabla p'\cdot\Lambda^s\u \,dx\Big|
&\lesssim\|\nabla p'\|_{L^\infty}\|\Lambda^sa\|_0\|\Lambda^s\u\|_0\lesssim\|a\|_{3}\|(a,\u)\|^2_s.
\end{split}\end{equation*}
Therefore, we have
\begin{equation*}\begin{split}
\frac12\frac{d}{dt}\int\frac{p'}{\rho}|\Lambda^sa|^2\,dx+\textrm{I}_3\lesssim\|(a,\u)\|_3\|(a,\u)\|_s^2.
\end{split}\end{equation*}
For $\textrm{I}_2$,$\textrm{I}_4$ and $\textrm{I}_{10}$, it follows from
Lemma \ref{Hk-estimate}, Lemma \ref{Hk-estimate-7-8}, Lemma \ref{Hk-estimate-1}, \eqref{pri-2}, and the embedding relation, that
\begin{align*}
|\textrm{I}_{2}|&
\leq\|[\Lambda^s,\rho \u\cdot\nabla]\u\|_0\|\Lambda^s\u\|_0\\
&\lesssim(1+\|(a,\u)\|_2)\|(a,\u)\|_3\|(a, \u)\|^2_s+\|\u\|_2\|\u\|_3\|\u\|_s,
\end{align*}
\begin{align*}
|\textrm{I}_{4}|&
\leq\|[\Lambda^s,p']\nabla\rho\|_0\|\Lambda^s \u\|_0\\
&\lesssim(\|\nabla a\|_{L^\infty}\|\Lambda^sa\|_0+\|\Lambda^s(p'(1+a)-p'(1))\|_{0}\|\nabla a\|_{L^\infty})\| \u\|_s\\
&\lesssim\|a\|_{3}\|(a,\u)\|^2_s,
\end{align*}
and
\begin{align*}
|\textrm{I}_{10}|&
\leq\|[\Lambda^s,\rho]\u\|_0\|\Lambda^s\u\|_0\\
&\lesssim(\|\u\|_{L^\infty}\|\Lambda^sa\|_0+\|\Lambda^{s-1}\u\|_{0}\|\nabla a\|_{L^\infty})\|\u\|_s\\
%\lesssim&(\|\nabla a\|_{L^\infty}\|\Lambda^sa\|_0+\|\Lambda^s(p'(1+a)-p'(1))\|_{0}\|Da\|_{L^\infty})\|u\|_s\\
&\lesssim\|(a,\u)\|_{3}\|(a,\u)\|^2_s.
\end{align*}
We now deal with terms $\textrm{I}_{5}$, $\textrm{I}_{6}$, $\textrm{I}_{8}$ and $\textrm{I}_{9}.$ First,  employing  the integration by
parts formula and  $\div \b=0$, using Lemma \ref{Hk-estimate}, %\eqref{EQUIVN},
and the embedding relation, we infer that
\begin{equation*}\begin{split}
|\textrm{I}_{5}+\textrm{I}_{8}|&\leq\Big|\int[\Lambda^s,\b\cdot\nabla] \b\cdot\Lambda^s \u \,dx+\int[\Lambda^s,\b\cdot\nabla] \u\cdot\Lambda^s \b \,dx\Big|
\\&\quad+\Big|\int \b\cdot\nabla\Lambda^s\b\cdot\Lambda^s \u \,dx+\int \b\cdot\nabla\Lambda^s \u\cdot\Lambda^s\b \,dx\Big|
\\&\lesssim\big(\|\nabla \u\|_{L^\infty}\| \b\|_s+\|\nabla \b\|_{L^\infty}\|\u\|_s\big )\|\u\|_s\\&\lesssim\|(\u,\b)\|_{3}\|(\u,\b)\|^{2}_s.
\end{split}\end{equation*}
It follows from the integration by parts formula, Lemma \ref{Hk-estimate}, %\eqref{EQUIVN},
and the embedding relation,  that
\begin{equation*}\begin{split}
  |\textrm{I}_{6}+\textrm{I}_{9}|
  &=\Big|-\frac12\int\Lambda^s\nabla|\b|^2\cdot\Lambda^s\u \,dx-\int\Lambda^s(\b\div \u)\cdot\Lambda^s\b \,dx\Big|
  \\&=\Big|-\sum_{i,k=1}^{d}\int\Lambda^s(\partial_{i}\b^k\b^k)\Lambda^s\u^{i} \,dx-\int\Lambda^s(\b\div \u)\cdot\Lambda^s \b \,dx\Big|
  \\&=\Big|-\sum_{i,k=1}^{d}\int[\Lambda^s,\b^k]\partial_{i}\b^k\cdot\Lambda^s \u^{i}\,dx-\int[\Lambda^s,\b]\div \u\cdot\Lambda^s \b \,dx\\&\quad-\sum_{i,k=1}^{d}\int\Lambda^s\partial_{i}\b^k\b^k\Lambda^s\u^{i}\,dx-\int \b\Lambda^s\div \u\cdot\Lambda^s \b \,dx\Big|
  \\&=\Big|-\sum_{i,k=1}^{d}\int[\Lambda^s,\b^k]\partial_{i}\b^k\cdot\Lambda^s\u^{i}\,dx-\int[\Lambda^s,\b]\div \u\cdot\Lambda^s \b \,dx\\&\quad+\sum_{i,k=1}^{d}\int\Lambda^s\b^k\partial_{i}\b^k\Lambda^s\u^{i}\,dx+\sum_{i,k=1}^{d}\int\Lambda^s\b^k\b^k\Lambda^s\partial_{i}\u^{i}\,dx
  -\int \b\Lambda^s\div \u\cdot\Lambda^s \b \,dx\Big|
  \\
  &=\Big|-\sum_{i,k=1}^{d}\int[\Lambda^s,\b^k]\partial_{i}\b^k\cdot\Lambda^s\u^{i}\,dx-\int[\Lambda^s,\b]\div \u\cdot\Lambda^s \b \,dx+\sum_{i,k=1}^{d}\int\Lambda^s\b^k\partial_{i}\b^k\Lambda^s\u^{i}\,dx\Big|
  \\&\lesssim\big(\|\nabla \u\|_{L^\infty}\| \b\|_s+\|\nabla \b\|_{L^\infty}\|\u\|_s\big)\|\u\|_s+\|\nabla \b\|_{L^\infty}\|\u\|_s\|\b\|_s\\&\lesssim\|(\u,\b)\|_{3}\|(\u,\b)\|^{2}_s.
\end{split}\end{equation*}
For the term $I_7$, applying the  integrating by parts,  Lemma \ref{Hk-estimate}, and using the embedding relation, we infer that
\begin{equation*}\begin{split}
|\textrm{I}_{7}|
%&=\Big|\int\Lambda^s(\u\cdot\nabla \b)\cdot\Lambda^s\b \,dx\Big|
%\\&=\Big|\int[\Lambda^s, \u\cdot\nabla]\b\cdot\Lambda^s \b \,dx+\int u\cdot\nabla\Lambda^s \b\cdot\Lambda^s \b \,dx\Big|\\
%\\
%&=\Big|\int[\Lambda^s, \u\cdot\nabla]\b\cdot\Lambda^s \b \,dx-\int \div \u|\Lambda^s \b|^{2}\,dx\Big|
%\\
&\lesssim\big(\|\nabla \u\|_{L^\infty}\| \b\|_s+\|\nabla \b\|_{L^\infty}\|\u\|_s\big)\|\b\|_s+\|\nabla \u\|_{L^\infty}\|\b\|^{2}_s\\&\lesssim\|(\u,\b)\|_{3}\|(\u,\b)\|^{2}_s,
\end{split}\end{equation*}
Finally, it remains to estimate $\textrm{I}_{1}.$ We first bound $\|\u_t\|_{L^\infty}.$ By \eqref{pri-2}
and the embedding relation,  we have
\begin{equation}\begin{split}\label{pri-5}
\|\partial_t\u\|_{L^\infty}&=\Big\|-(\u\cdot\nabla)\u-\u-\frac{p'}{\rho}\nabla a+\rho^{-1}\big((\omega\cdot \nabla) \b+(\b\cdot \nabla) \b-\nabla(\omega\cdot \b)-\frac12\nabla(|\b|^2)\big)\Big\|_{L^\infty}\\
&\lesssim\|\u\|_{L^\infty}\|\nabla\u\|_{L^\infty}+\|\u\|_{L^\infty}+\|\nabla a\|_{L^\infty}+(1+\|\b\|_{L^\infty})\|\nabla\b\|_{L^\infty}\\
&\lesssim\big(1+\|(\u,\b)\|_{2}\big)\|(a,\u,\b)\|_3.
\end{split}\end{equation}
Next, we estimate $\|\u_t\|_{s-1}.$ Note that \eqref{pri-2} implies $a\in[-\frac12,\frac12]$. Thus, it follows from Lemma \ref {Hk-estimate-1}
that
\begin{equation*}\begin{split}
\Big\|\frac{p'}{\rho}\Big\|_{s-1}&=\Big\|\frac{p'(a+1)}{a+1}-p'(1)+p'(1)\Big\|_{s-1}\lesssim\Big\|\frac{p'(a+1)}{a+1}-p'(1)\Big\|_{s-1}+1\\
&\lesssim\Big\|\frac{p'(a+1)}{a+1}-p'(1)\Big\|_{s}+1\lesssim\|a\|_{s}+1
\end{split}\end{equation*}
and
\begin{equation*}\begin{split}
\Big\|\frac{1}{\rho}\Big\|_{s-1}&=\|(a+1)^{-1}-1+1\|_{s-1}\lesssim\|(a+1)^{-1}-1\|_{s-1}+1\\
&\lesssim\|(a+1)^{-1}-1\|_{s}+1\lesssim\|a\|_{s}+1.
\end{split}\end{equation*}
Thanks to the above two estimates, and $\rho\in[\frac12,\frac32]$, it follows from Lemma \ref{Hk-estimate}, embedding relation and $s\geq3$ that
\begin{equation}\begin{split}\label{pri-6}
\|\partial_t\u\|_{s-1}&=\left\|-\u\cdot\nabla \u-\u-\frac{p'}{\rho}\nabla a+\rho^{-1}\big((\omega\cdot \nabla) \b+(\b\cdot \nabla) \b-\nabla(\omega\cdot \b)-\frac12\nabla(|\b|^2)\big)\right\|_{s-1}\\
&\lesssim\|\u\|_{L^\infty}\|\nabla\u\|_{s-1}+\|\u\|_{s-1}\|\nabla\u\|_{L^\infty}+\|\u\|_{s-1}+\Big\|\frac{p'}{\rho}\Big\|_{s-1}\|\nabla a\|_{L^\infty}\\
&\quad+\Big\|\frac{p'}{\rho}\Big\|_{L^\infty}\|\nabla a\|_{s-1}+\Big\|\frac1\rho\Big\|_{L^\infty}
\|(\omega\cdot \nabla) \b+(\b\cdot \nabla) \b-\nabla(\omega\cdot \b)-\frac12\nabla(|\b|^2)\|_{s-1}\\
&\quad+\Big\|\frac1\rho\Big\|_{s-1}
\|(\omega\cdot \nabla) \b+(\b\cdot \nabla) \b-\nabla(\omega\cdot \b)-\frac12\nabla(|\b|^2)\|_{L^\infty}\\
&\lesssim\|\u\|_3\|\u\|_s+(1+\|a\|_s)\|a\|_3+\|(a,\u)\|_s+(1+\|\b\|_{L^\infty})\|\b\|_{s}\\
&\quad+(1+\|a\|_{s-1})\|\nabla \b\|_{L^\infty}
+(1+\|a\|_s)\|\nabla \b\|_{L^\infty}\|\b\|_{L^\infty}\\
&\lesssim\|\u\|_3\|\u\|_s+(1+\|a\|_s)\|a\|_3+\|(a,\u)\|_s+(1+\|\b\|_2)\|\b\|_{s}\\
&\quad+(1+\|a\|_{s})\|\b\|_3
+(1+\|a\|_s)\|\b\|_3^{2}
\\ &\lesssim\|\u\|_3\|\u\|_s+\|a\|_3\|a\|_s+\|(a,\u,\b)\|_s+\|\b\|_2\|\b\|_{s}\\
&\quad+\|\b\|^{2}_3+\|\b\|_3\|a\|_s+\|a\|_s\|\b\|^{2}_3.
\end{split}\end{equation}
This together with \eqref{pri-5}, \eqref{pri-6}, Lemma \ref{Hk-estimate-7-8}, and embedding relation yields
\begin{equation*}\begin{split}
|\textrm{I}_{1}|&=\left|\int[\Lambda^s,\rho]\partial_t\u\cdot\Lambda^s \u \,dx\right|
=\left|\int[\Lambda^s,a]\partial_t\u\cdot\Lambda^s \u \,dx\right|\\
&\leq\|[\Lambda^s,a]\partial_t\u\|_0\|\Lambda^s \u\|_0\lesssim
(\|\nabla a\|_{L^\infty}\|\partial_t\u\|_{s-1}+\|a\|_s\|\partial_t\u\|_{L^\infty})\|\u\|_s\\
&\lesssim\|a\|_3\|\u\|_s\big(\|\u\|_3\|\u\|_s+\|a\|_3\|a\|_s+\|(a,\u,\b)\|_s+\|\b\|_2\|\b\|_{s}\\
&\quad+\|\b\|^{2}_3+\|\b\|_3\|a\|_s+\|a\|_s\|\b\|^{2}_3\big)+(1+\|(\u,\b)\|_2)\|(a,\u,\b)\|_3\|a\|_s\|\u\|_s.
\end{split}\end{equation*}
Plugging the estimates for $\textrm{I}_{1}$--$\textrm{I}_{10}$ into \eqref{pri-4} and  combining with \eqref{pri-1-1} and %summing up for any $1\leq s\leq \ell$,
using \eqref{pri-2} and \eqref{2026-6-27-pri-1-1}, we finally obtain the desired inequality. This completes the proof of Lemma \ref{lemma-pri1}.
\end{proof}

Due to the lack of dissipation terms for the density and velocity field, one needs to
exploit the hidden dissipation benefit from the background magnetic
field. To this end, we rewrite system $\eqref{MHD1}$ as
\begin{equation}\label{MHD2}
\left\{
\begin{array}{rl}
&\hspace{-0,5cm}\partial_ta+\div \u=f_1,
\smallskip\\
&\hspace{-0,5cm}\partial_t \u
+\beta\nabla a +\u=(\omega\cdot \nabla) \b-\nabla(\omega\cdot  \b)+f_2,
\medskip\\
&\hspace{-0,5cm}\partial_t  \b
=(\omega\cdot \nabla)\u-\omega\div \u+f_3,
\medskip\\
&\hspace{-0,5cm}{\mathop{\rm div}}\, \b=0,
\medskip\\
&\hspace{-0,5cm}(\rho,\u, \b)_{|t=0}=(\rho_0,\u_0, \b_0),
\end{array}
\right.
\end{equation}
where $\beta:=p'(1)>0$, and $f_1,f_2,f_3$ are the nonlinear terms expressed as
\begin{equation*}\begin{split}
&f_1:=-\div (a \u),\\
&f_2:=-a\partial_t\u-(p'(a+1)-p'(1))\nabla a-a\u-\rho(\u\cdot\nabla)\u+( \b\cdot\nabla) \b-\frac{1}{2}\nabla|\b|^2,\\
&f_3:=-(\u\cdot\nabla) \b+( \b\cdot\nabla)\u- \b\div \u.
\end{split}\end{equation*}

\begin{lemma}\label{lemma-pri3-3-1}
Let $(a,\u,\b) \in C([0, T];H^N)$ be a solution to the  system \eqref{MHD2}. Then, for any $t\in[0,T]$, $s\geq0$, we have
\begin{equation}
\begin{split}\label{pri-199}
\beta\|&(\omega\cdot\nabla)\Lambda^{s}a\|_0^2+\frac{d}{dt}\int \omega\cdot\nabla\Lambda^{s}a\cdot\Lambda^{s}(\u\cdot\omega)\,dx\lesssim \|\u\|_{s+1}^{2}+\|(f_1,f_2)\|_{s}^{2}.
\end{split}
\end{equation}
\end{lemma}
\begin{proof}
%Repeating the similar estimate of  we conclude by \eqref{priori-1}, $L\geq3+r$ and $3+r<M< N-2-r$ that
Multiplying  $\eqref{MHD2}_{2}$ by $\omega$, applying $\Lambda^{s}$ to $\eqref{MHD2}_{2}$, and taking $L^2$ inner product to the resultant with $\omega\cdot\nabla\Lambda^{s}a$,
integrating by parts, and using $\eqref{MHD2}_{1}$ yield
\begin{equation}\begin{split}\label{2-pri-77-1}
\beta\|(\omega\cdot\nabla)\Lambda^{s}a\|_0^2
&=-\int (\omega\cdot\nabla)\Lambda^{s}a\cdot\Lambda^{s}\partial_t(\u\cdot\omega)\,dx-\int (\omega\cdot\nabla)\Lambda^{s}a\cdot\Lambda^{s}(\u\cdot\omega)\,dx\\&\quad+\int \omega\cdot\nabla\Lambda^{s}a \cdot\Lambda^{s}(f_2\cdot\omega)\,dx.
\end{split}\end{equation}
For the first term on the right-hand side of \eqref{2-pri-77-1}, one gets by $\eqref{MHD2}_1$ that
\begin{equation*}\begin{split}%\label{pri-17}
&-\int (\omega\cdot\nabla)\Lambda^{s}a\cdot\Lambda^{s}\partial_t(\u\cdot\omega)\,dx\\
&\quad =-\frac{d}{dt}\int (\omega\cdot\nabla)\Lambda^{s}a\cdot\Lambda^{s}(\u\cdot\omega)\,dx+\int (\omega\cdot\nabla)\Lambda^{s}\partial_ta\cdot\Lambda^{s}(\u\cdot\omega)\,dx\\
&\quad =-\frac{d}{dt}\int (\omega\cdot\nabla)\Lambda^{s}a\cdot\Lambda^{s}(\u\cdot\omega)\,dx+\int(\omega\cdot\nabla)\Lambda^{s}(-\text{div}\u+f_1)\cdot\Lambda^{s}(\u\cdot\omega)\,dx.
%&\ =-\frac{d}{dt}\int(\omega\cdot\nabla)\Lambda^Lh\cdot\Lambda^Lu \,dx+\textrm{III}_1+\textrm{III}_2,
\end{split}\end{equation*}
%Obviously, using integration by
%parts formula and  $\text{div}h=0$, we have
%$$
%-\int(\omega\cdot\nabla)\Lambda^Lh\cdot\nabla\Lambda^L(\beta a+ h\cdot\omega)\,dx=\int(\omega\cdot\nabla)\Lambda^L\text{div}h\Lambda^L(\beta a+ h\cdot\omega)\,dx=0.
%$$
It follows from H\"{o}lder's and Young's inequalities that
\begin{align*}
 \Big|\int(\omega\cdot\nabla)\Lambda^{s}(-\text{div}\u+f_1)\cdot\Lambda^{s}(\u\cdot\omega)\,dx\Big| &\lesssim\|(\omega\cdot\nabla)\Lambda^{s}\text{div}\u\|^{2}_0+\|(\omega\cdot\nabla)\Lambda^{s}\text{div}\u\|_0\|\omega\cdot\Lambda^{s}f_1\|_{0}\\
  &\lesssim\|\u\|^2_{s+1}+\|f_1\|^{2}_{s},
\end{align*}
%\begin{align*}&\Big|(\omega\cdot\nabla)\Lambda^{s}a\cdot\Lambda^{s}(\omega\cdot\nabla(\b\cdot\omega))+\int(\omega\cdot\nabla)\Lambda^{s}a\cdot\Lambda^{s}(\omega\cdot\nabla \b\cdot\omega)\,dx\Big|\lesssim \|\b\|^2_{s+1}+\epsilon\|(\omega\cdot\nabla)\Lambda^{s}a\|_{0}^{2},\end{align*}
\begin{align*}
  \big|\int (\omega\cdot\nabla)\Lambda^{s}a\cdot\Lambda^{s}(\u\cdot\omega)\,dx\big|\leq C_{\epsilon_1} \|\u\|^2_{s}+\epsilon_1\|(\omega\cdot\nabla)\Lambda^{s}a\|_{0}^{2},
\end{align*}
and
\begin{equation*}
\begin{split}
\big|\int(\omega\cdot\nabla)\Lambda^{s}a\cdot\Lambda^{s} f_2 \,dx\big|\lesssim &\|(\omega\cdot\nabla)\Lambda^{s}a\|_{0}\|\Lambda^{s}f_2\|_0
\leq C_{\epsilon_1} \|f_2\|_{s}^{2}+\epsilon_1\|(\omega\cdot\nabla)\Lambda^{s}a\|_{0}^{2}.
\end{split}
\end{equation*}
Substituting the above estimates into \eqref{2-pri-77-1} and then taking $\epsilon_1$ small enough, we conclude that \eqref{pri-199} holds. This completes the proof of Lemma \ref{lemma-pri3-3-1}.
\end{proof}
\vskip .1in
\begin{lemma}\label{lemma-pri3-3}
Let $(a,\u,\b) \in C([0, T];H^N)$ be a solution to the  system \eqref{MHD2}.  Then, for any $t\in[0,T]$, $s\geq0$, we have
\begin{equation}
\begin{split}\label{pri-199-9LLL}
\|&(\omega\cdot\nabla)\Lambda^{s}\b\|_0^2-\frac{d}{dt}\int\Lambda^{s}\u\cdot(\omega\cdot\nabla)\Lambda^{s}\b \,dx\lesssim \|\u\|_{s+1}^{2}+\|(f_2,f_3)\|_{s}^{2}.
\end{split}
\end{equation}
\end{lemma}
\begin{proof}
%Repeating the similar estimate of  we conclude by \eqref{priori-1}, $L\geq3+r$ and $3+r<M< N-2-r$ that
Applying $\Lambda^{s}$ to $\eqref{MHD2}_{2}$, taking the  $L^2$ inner product of the resultant with $\omega\cdot\nabla\Lambda^{s}\b$,
integrating by parts, and using $\eqref{MHD2}_{1}$ yield
\begin{equation}\begin{split}\label{2-pri-77}
\|\omega\cdot\nabla\Lambda^{s}\b\|_0^2
&=\int (\omega\cdot\nabla)\Lambda^{s}\b\cdot\Lambda^{s}\partial_t\u \,dx-\int(\omega\cdot\nabla)\Lambda^{s}\b\cdot\nabla\Lambda^{s}(\beta a+ \b\cdot\omega)\,dx
\\&\quad-\int (\omega\cdot\nabla)\Lambda^{s}\b\cdot\Lambda^{s}\u \,dx+\int (\omega\cdot\nabla)\Lambda^{s}\b\cdot\Lambda^{s}f_2\,dx.
\end{split}\end{equation}
For the first term on the right-hand side of \eqref{2-pri-77}, we conclude  by $\eqref{MHD2}_3$ that
\begin{equation*}\begin{split}\label{pri-17}
&\int(\omega\cdot\nabla)\Lambda^{s}\b\cdot\Lambda^{s}\partial_t\u \,dx
\\
&\quad=\frac{d}{dt}\int(\omega\cdot\nabla)\Lambda^{s}\b\cdot \Lambda^{s}\u \,dx-\int(\omega\cdot\nabla)\Lambda^{s}\b_t\cdot\Lambda^{s}\u \,dx\\
&\quad=\frac{d}{dt}\int(\omega\cdot\nabla)\Lambda^{s}\b\cdot \Lambda^{s}\u \,dx-\int(\omega\cdot\nabla)\Lambda^{s}(\omega\cdot\nabla \u-\omega\text{div}\u+f_3)\Lambda^{s}\u \,dx.
%&\ =-\frac{d}{dt}\int(\omega\cdot\nabla)\Lambda^Lh\cdot\Lambda^Lu \,dx+\textrm{III}_1+\textrm{III}_2,
\end{split}\end{equation*}
Obviously, using the integration by
parts formula and  $\text{div}\b=0$, we have
$$
-\int(\omega\cdot\nabla)\Lambda^{s}\b\cdot\nabla\Lambda^{s}(\beta a+ \b\cdot\omega)\,dx=\int\big((\omega\cdot\nabla)\Lambda^{s}(\text{div}\b)\big)\Lambda^{s}(\beta a+ \b\cdot\omega)\,dx=0.
$$
It follows from H\"{o}lder's and Young's inequalities, that
\begin{align*}
 \big|\int(\omega\cdot\nabla)\Lambda^{s}\b\cdot\Lambda^{s}\u \,dx\big|\lesssim&\|(\omega\cdot\nabla)\Lambda^{s}\b\|_0\|\u\|_{s}
  \leq C_{\epsilon_1}\|\u\|^2_{s}+\epsilon_1\|(\omega\cdot\nabla)\Lambda^{s}\b\|_0^2,
\end{align*}
\begin{align*}
  \big|\int(\omega\cdot\nabla)\Lambda^{s}(\omega\cdot\nabla \u-\omega\text{div}\u+f_3)\Lambda^{s}u\,dx\big|\lesssim \|\u\|^2_{s+1}+\|f_3\|^2_{s},
\end{align*}
and
\begin{equation*}
\begin{split}
\big|\int(\omega\cdot\nabla)\Lambda^{s}\b\cdot\Lambda^{s} f_2 \,dx\big|\lesssim &\|(\omega\cdot\nabla)\Lambda^{s}\b\|_{0}\|\Lambda^{s}f_2\|_0
\leq C_{\epsilon_1} \|f_2\|_{s}^{2}+\epsilon_1\|(\omega\cdot\nabla)\Lambda^{s}\b\|_{0}^{2}.
\end{split}
\end{equation*}
Substituting the above estimates into \eqref{2-pri-77}, and then taking $\epsilon_1$ small enough, we conclude that \eqref{pri-199-9LLL} holds.
This completes the proof of Lemma \ref{lemma-pri3-3}.
\end{proof}

%The following lemma exploits a new dissipating mechanism for the perturbed density $a$ which is generated from the interaction
%between the velocity field and the constant background magnetic field $w$.

\begin{lemma}\label{lemma-pri2}Suppose $\omega$ satisfies the Diophantine condition \eqref{Diophantine0}
with constants $c$ and $r$. Let $(a,\u,\b) \in C([0, T];H^N)$ be a solution to the  system \eqref{MHD2}.
Then, it holds for any $t\in[0,T]$, $s\geq r$, that
\begin{equation}
\begin{split}\label{pri-199-9LLLL}
&c_1\|(a,\b)\|_{s-1}^2+\frac{d}{dt}\int \Big(\omega\cdot\nabla\Lambda^{s}a\cdot\Lambda^{s}(\u\cdot\omega)-\Lambda^{s}\u\cdot(\omega\cdot\nabla)\Lambda^{s}\b\Big) \,dx\\
&\quad\lesssim \|\u\|_{s+1}^{2}+\|(f_1,f_2,f_3)\|_{s}^{2}
\end{split}
\end{equation}
with a positive constant $c_1$ depending only on $\beta,c,r,|\omega|$ and $|\mathbb T^3|.$
\end{lemma}

\begin{proof} For the density $a$, it follows from \eqref{junzhi1} that
$\int a\,dx=0$. Then %summing up for any $0\leq s\leq s_1$ in \eqref{pri-199}, and
using \eqref{pri-199} and Lemma \ref{Diophantine-inequality}, we  get
\begin{equation*}
\begin{split}\label{pri-199-99LLL}
& c_0\beta \|a\|_{s-r}^2+\frac{d}{dt}\int \omega\cdot\nabla\Lambda^{s}a\cdot\Lambda^{s}(\u\cdot\omega)\,dx\lesssim \|\u\|_{s+1}^{2}+\|(f_2,f_3)\|_{s}^{2}
\end{split}
\end{equation*}
with a positive constant $c_0$ depending only on $c,r,|\omega|$ and $|\mathbb T^3|.$ Similarly, for the magnetic field $\b$,  %summing up for any $0\leq s\leq s_2$ in \eqref{pri-199-9LLL}, employing \eqref{EQUIVN}
recalling $\int \b \,dx=0$ in  (\ref{junzhi1}) and using  \eqref{pri-199-9LLL} and Lemma \ref{Diophantine-inequality}, we also  infer that
\begin{equation}
\begin{split}\label{pri-199-9LLL-LL}
&c_0\|\b\|_{s-r}^2-\frac{d}{dt}\int\Lambda^{s}\u\cdot(\omega\cdot\nabla)\Lambda^{s}\b \,dx\lesssim \|\u\|_{s+1}^{2}+\|(f_2,f_3)\|_{s}^{2}.
\end{split}
\end{equation}
%Thus, taking $s_2=s_1+r+1$,  we conclude  from \eqref{pri-199-9LLL-LL}, that
%\begin{equation}\begin{split}\label{pri-199-99}\|\b\|_{s_1+1}^2-\frac{d}{dt}\sum_{0\leq s\leq s_1+r+1}\int\Lambda^{s}\u\cdot(\omega\cdot\nabla)\Lambda^{s}\b \,dx\lesssim \|\u\|_{s_1+r+2}^{2}+\|(f_2,f_3)\|_{s_1+r+1}^{2}.\end{split}\end{equation}
 %Multiplying \eqref{pri-199-99LLL} by sufficiently small positive constant,  adding up to \eqref{pri-199-99},  and then  using  the embedding relation,  we finally deduce that \eqref{pri-199-9LLLL} holds.
 Thus, the desire conclusion follows the above two inequalities. This completes the proof of Lemma \ref{lemma-pri2}.
\end{proof}

\vskip .3in
\section{The proof of Theorem \ref{dingli}}
\label{sec:pro}

This section completes the proof of Theorem \ref{dingli}. The framework of the proof is the
bootstrapping argument.

\begin{proof}[Proof of Theorem \ref{dingli}]  First of all, the MHD system in \eqref{MHD1} with any
initial data $(a_0, \u_0,  \b_0)$ in $ H^N(\T^3)$ has a unique local solution.
This follows from a standard contraction mapping argument (see, e.g.,
\cite{MM}). The
bootstrapping argument is employed to prove the global existence and stability.
It starts with the ansatz that the solution $(a(t), \u(t),
\b(t))$ to \eqref{MHD1} satisfies
\begin{align}\label{ping23}
\sup_{t\in[0,T]}(\|a\|_{{N}}+\|\u\|_{{N}}+\|\b\|_{{N}})\le \delta,
\end{align}
for some $0<\delta<1$ to be determined later. We then show that
\begin{align*}%\label{ping231}
\sup_{t\in[0,T]}(\|a\|_{{N}}+\|\u\|_{{N}}+\|\b\|_{{N}})\le \frac{\delta}{2}.
\end{align*}
The a priori assumption \eqref{pri-2} is guaranteed by the embedding relation and the smallness of the norm stated in \eqref{ping23}.
On the one hand, taking $s=r+1$ in \eqref{pri-2-2}, yields that
\begin{equation}\label{pri-2-2-22}
\begin{split}
\frac{d}{dt}E_{r+1}+\|(\u,\Lambda^{r+1}\u)\|_{0}^2
%&\quad\lesssim\|a\|_3\|\u\|_{4r+2}\big(\|\u\|_3\|\u\|_{4r+2}+\|a\|_3\|a\|_{4r+2}+\|(a,\b)\|_3+\|(a,\b)\|_{4r+2}\\
%&\qquad+\|\b\|_2\|\b\|_{4r+2}+\|\b\|^{2}_3+\|\b\|_3\|a\|_{4r+2}+\|a\|_{4r+2}\|\b\|^{2}_3\big)\\
%&\qquad+(1+\|(a,\u,\b)\|_3)\|(a,\u,\b)\|_3\|(a,\u,\b)\|^2_{4r+2}\\
\lesssim(1+\|(a,\u,\b)\|^2_3)\|(a,\u,\b)\|_3\|(a,\u,\b)\|^2_{r+1}.
\end{split}\end{equation}
On the other hand,  choosing  $s=r$ in \eqref{pri-199-9LLLL} implies that
\begin{equation}
\begin{split}\label{pri-199-9LLLL-L}
&c_1\|(a,\b)\|_{r-1}^2+\frac{d}{dt}\int \Big(\omega\cdot\nabla\Lambda^{r}a\cdot\Lambda^{r}(\u\cdot\omega)-\Lambda^{r}\u\cdot(\omega\cdot\nabla)\Lambda^{r}\b\Big) \,dx\\&\quad\leq C\|\u\|_{r+1}^{2}+C\|(f_1,f_2,f_3)\|_{r}^{2}.
\end{split}
\end{equation}
Hence, let
 $A$ be a positive  constant determined later. We infer from \eqref{pri-199-9LLLL-L} and \eqref{pri-2-2-22} that
\begin{equation}\label{pri-2-2-22-2A}
\begin{split}
&\frac{d}{dt}\left\{AE_{r+1}(t)+\int \Big(\omega\cdot\nabla\Lambda^{r}a\cdot\Lambda^{r}(\u\cdot\omega)-\Lambda^{r}\u\cdot(\omega\cdot\nabla)\Lambda^{r}\b\Big) \,dx\right\}\\
&\quad\quad+A\|(\u,\Lambda^{r+1}\u)\|_{0}^2-C\|\u\|_{r+1}^2+c_1\|(a,\b)\|_{r-1}^2\\
&\quad\lesssim (1+\|(a,\u,\b)\|^2_3)\|(a,\u,\b)\|_3\|(a,\u,\b)\|^2_{r+1}+\|(f_1,f_2,f_3)\|_{r}^{2}.
\end{split}\end{equation}
Define
\begin{align*}%\label{ping31}
{\mathcal{E}(t)}=&AE_{r+1}(t)+\int \Big(\omega\cdot\nabla\Lambda^{r}a\cdot\Lambda^{r}(\u\cdot\omega)-\Lambda^{r}\u\cdot(\omega\cdot\nabla)\Lambda^{r}\b\Big) \,dx.
\end{align*}
Then, from \eqref{bu-2026-6-27-1} and the
Cauchy--Schwarz inequality,  we may choose $A$ sufficiently large, depending only on $r$ and $\omega$ such that
\begin{equation}\label{bu2026-6-28-1}c_1\|\u\|_{r+1}^2\leq A\|(\u,\Lambda^{r+1}\u)\|_{0}^2-C\|\u\|^2_{r+1}
\end{equation} and
\begin{equation}\label{bu2026-6-28-2}
\|(a,\u,\b)\|^2_{r+1}\leq {\mathcal{E}(t)}\leq 2A\|(a,\u,\b)\|^2_{r+1}.
\end{equation}
We infer from \eqref{pri-2-2-22-2A} and \eqref{bu2026-6-28-1} that
\begin{equation}\label{bu-pri-2-2-22-2A}
\begin{split}
&\frac{d}{dt}{\mathcal{E}(t)}+c_1\|(a,\u,\b)\|_{0}^2\lesssim (1+\|(a,\u,\b)\|^2_3)\|(a,\u,\b)\|_3\|(a,\u,\b)\|^2_{r+1}+\|(f_1,f_2,f_3)\|_{r}^{2}.
\end{split}\end{equation}
It follows from Lemma \ref{Hk-estimate} and $\|\cdot\|_{L^\infty}\lesssim\|\cdot\|_2$ that
\begin{equation}\begin{split}\label{4.6AAAA}
\|f_1\|_{r}\lesssim\|a\u\|_{r+1}\lesssim\|(a,\u)\|_{L^\infty}\|(a,\u)\|_{r+1}\lesssim\|(a,\u)\|_{2}\|(a,\u)\|_{r+1},
\end{split}\end{equation}
and
\begin{equation}\begin{split}\label{bu4.7AAAA}
\|f_3\|_{r}\lesssim&\|(\u\cdot\nabla) \b\|_{r}+\|( \b\cdot\nabla)\u\|_{r}+\| \b\div \u\|_{r}\\
\lesssim&\|(\u,\nabla \b)\|_{L^\infty}\|(\u,\nabla  \b)\|_{r}+\| (\b,\nabla \u)\|_{L^\infty}\| (\b,\nabla \u)\|_r\\
\lesssim&\|(\u, \b)\|_3\|(\u, \b)\|_{r+1}.
\end{split}\end{equation}
It remains to bound the term $f_{2}.$  We  deduces by \eqref{pri-5}, \eqref{pri-6}, Lemma \ref{Hk-estimate},
and the embedding relation, that
\begin{equation*}
\begin{split}
\|a\partial_t\u\|_{r}&\lesssim\|a\|_{L^\infty}\|\partial_t\u\|_{r}+\|a\|_{r}\|\partial_t\u\|_{L^\infty}\\
&\lesssim \|a\|_3\big(\|\u\|_3\|\u\|_{r+1}+\|a\|_3\|a\|_{r+1}+\|(a,\u,\b)\|_{r+1}+\|\b\|_2\|\b\|_{r+1}\\
&\quad+\|\b\|^{2}_3+\|\b\|_3\|a\|_{r+1}+\|a\|_{r+1}\|\b\|^{2}_3\big)+\|a\|_{r}\big(1+\|\u,\b\|_{2}\big)\|(a,\u,\b)\|_3\\
&\lesssim (1+\|(a,\u,\b)\|_3)\|(a,\u,\b)\|_3\|(a,\u,\b)\|_{r+1}.
\end{split}
\end{equation*}
It follows from \eqref{pri-2}, Lemma \ref{Hk-estimate-1}, and the embedding inequality that
\begin{equation*}\begin{split}
\|(p'(a+1)-p'(1))\nabla a\|_{r}
&\lesssim\|p'(a+1)-p'(1)\|_{L^\infty}\|a\|_{r+1}+\|p'(a+1)-p'(1)\|_{r}\|\nabla a\|_{L^\infty}\\
&\lesssim\|a\|_{L^\infty}\|a\|_{r+1}+\|a\|_{r}\|\nabla a\|_{L^\infty}
\\&\lesssim\|a\|_3\|a\|_{r+1}.
\end{split}\end{equation*}
Employing  Lemma \ref{Hk-estimate} and \eqref{pri-2} yields that
\begin{equation*}\begin{split}
&\left\|-\rho(\u\cdot\nabla)\u-a\u+(\b\cdot\nabla)\b-\frac{1}{2}\nabla|\b|^2\right\|_{r}\\
&\quad\lesssim \|\rho \u,\nabla\u\|_{L^\infty}\|\rho \u,\nabla\u\|_{r}+\|(a,\u)\|_{L^\infty}\|(a,\u)\|_{r}+\|(\b,\nabla\b)\|_{L^\infty}\|(\b,\nabla\b)\|_{r}\\
&\quad\lesssim(\|\rho\|_{L^\infty}\|\u\|_{L^\infty}+\|\nabla\u\|_{L^\infty})(\|\rho\|_{L^\infty}\|\u\|_{r}+\|\rho\|_{r}\|\u\|_{L^\infty}
+\|\u\|_{r+1})
+\|(a,\u,\b)\|_3\|(a,\u,\b)\|_{r+1}\\
&\quad\lesssim\|\u\|_3(\|\u\|_{r+1}+\|\rho\|_{r}\|\u\|_{3})+\|(a,\u,\b)\|_3\|(a,\u,\b)\|_{r+1}\\
&\quad\lesssim\|(a,\u,\b)\|_3\|(a,\u,\b)\|_{r+1}+\|\u\|_3^2(1+\|a\|_{r})
\\&\quad\lesssim (1+\|(a,\u,\b)\|_3)\|(a,\u,\b)\|_3\|(a,\u,\b)\|_{r+1}.%\lesssim&(1+\|(a,\u),h\|_3)\|(a,\u),h\|_3\|(a,\u),h\|_{m+1}.
\end{split}\end{equation*}
Combining  with the above three inequalities,  we have
\begin{equation}\begin{split}\label{4.8AAAA}
\|f_2\|_{r}\lesssim (1+\|(a,\u,\b)\|_3)\|(a,\u,\b)\|_3\|(a,\u,\b)\|_{r+1}.
\end{split}\end{equation}
Then inserting \eqref{4.6AAAA}-\eqref{4.8AAAA} into \eqref{bu-pri-2-2-22-2A} gives rise to
\begin{equation}\label{pri-2-2-22-2}
\begin{split}
\frac{d}{dt}{\mathcal{E}(t)}+c_1\|(a,\u,\b)\|_{0}^2\lesssim \big(1+\|(a,\u,\b)\|_3\big)^3\|(a,\u,\b)\|_3\|(a,\u,\b)\|^2_{r+1}.
\end{split}\end{equation}
For any ${N}>3r+3$, by the interpolation inequality and the embedding relation, it holds that
\begin{equation*}\begin{split}
\|(a,\u,\b)\|^3_{r+1}&\lesssim \|(a,\u,\b)\|_{0}^{\frac{3(N-r-1)}{N}}\|(a,\u,\b)\|_{N}^{\frac{3(r+1)}{N}}
\\&\lesssim\|(a,\u,\b)\|_{0}^{2}\|(a,\u,\b)\|_{0}^{\frac{3(N-r-1)}{N}-2}\|(a,\u,\b)\|_{N}^{\frac{3(r+1)}{N}}
\\&\lesssim\|(a,\u,\b)\|_{0}^{2}\|(a,\u,\b)\|_{N}^{\frac{3(N-r-1)}{N}-2}\|(a,\u,\b)\|_{N}^{\frac{3(r+1)}{N}}
\\&\lesssim\|(a,\u,\b)\|_{0}^{2}\|(a,\u,\b)\|_{N},
\end{split}\end{equation*}
which together with \eqref{ping23} and the embedding relation yields that
\begin{equation*}\label{pri-2-2-22-4}
\begin{split}\big(1+\|(a,\u,\b)\|_3\big)^3\|(a,\u,\b)\|_3\|(a,\u,\b)\|^2_{r+1}
&\lesssim \big(1+\|(a,\u,\b)\|_3\big)^3\|(a,\u,\b)\|^3_{r+1}
\\&\lesssim\big(1+\|(a,\u,\b)\|_3\big)^3\|(a,\u,\b)\|_{N}\|(a,\u,\b)\|_{0}^{2}
\\&\lesssim\delta\|(a,\u,\b)\|_{0}^{2}.
\end{split}\end{equation*}
Plugging the above estimate  into \eqref{pri-2-2-22-2}, and then taking $\delta$ small enough such that $C\delta\leq\frac{c_1}{2}$,  we  deduce that
\begin{equation}\label{pri-2-2-22-6}
\frac{d}{dt}{\mathcal{E}(t)}+\frac{c_1}{2}\|(a,\u,\b)\|_{0}^2\leq0.
\end{equation}
For any ${N}> 3r+3$, by \eqref{bu2026-6-28-2}, the interpolation inequality and the  assumption (\ref{ping23}),  we have
\begin{equation*}%\label{pri-2-2-22-7}
 \big(\mathcal{E}(t)\big)^{\frac{N}{N-r-1}}\leq \big(2A\|(a,\u,\b)\|^2_{r+1}\big)^{\frac{N}{N-r-1}}\lesssim\|(a,\u,\b)\|_{0}^{2}\|(a,\u,\b)\|_{N}^{\frac{2r+2}{N-r-1}}
\leq C\delta^{\frac{2r+2}{N-r-1}}\|(a,\u,\b)\|_{0}^{2},
\end{equation*}
which along with \eqref{pri-2-2-22-6} yields the following Lyapunov-type inequality
\begin{equation*}%\label{pri-2-2-22-8-1}
\frac{d}{dt}{\mathcal{E}(t)}+c\mathcal{E}(t)^{\frac{N}{N-r-1}}\leq0.
\end{equation*}
It then follows easily that
\begin{align}\label{pri-2-2-22-8}
{\mathcal{E}(t)}\le C(1+t)^{-\frac{N-r-1}{r+1}}.
\end{align}
Taking $s={N}$ in \eqref{pri-2-2} and using $\eqref{bu-2026-6-27-1}_1$ and the fact, $\|\cdot\|^2_{N}\approx\|\cdot\|^2_{0}+\|\Lambda^N(\cdot)\|^2_{0},$ give rise to
\begin{equation}\label{pri-2-2-22-9}
\begin{split}
\frac{d}{dt}E_{N}+c_2\|\u\|_{N}^2&\leq C(1+\|(a,\u,\b)\|^2_3)\|(a,\u,\b)\|_3\|(a,\u,\b)\|^2_{N}\\
&\leq C(1+\|(a,\u,\b)\|^2_3)\|(a,\u,\b)\|_3E_{N}.
\end{split}\end{equation}
From \eqref{pri-2-2-22-8}, \eqref{bu2026-6-28-2}, and $N>3r+3$, we have
\begin{align*}
\int_0^T(1+\|(a,\u,\b)\|^2_3)\|(a,\u,\b)\|_3\,d\tau\lesssim\int_0^T(1+t)^{-\frac{N-r-1}{2r+2}}\,dt\le C.
\end{align*}
Then, applying Gronwall's inequality to \eqref{pri-2-2-22-9} and using \eqref{bu-2026-6-27-1} and \eqref{2026-6-27-3-1} yield
\begin{equation}\begin{split}\label{bu-pri-2-2-22-10}
\big\|(a,\u,\b)\big\|^2_{N}+\int_0^t\|\u(\tau)\|_{N}^2\,d\tau&\lesssim E_{N}(t)+c_2\int_0^T\|\u(\tau)\|_{N}^2d\tau\\
&\lesssim E_{N}(0)\lesssim\big\|(a(0),\u(0),\b(0))\big\|^2_{N}\le C\varepsilon^2.\end{split}\end{equation}
Furthermore, taking $\varepsilon$ small enough in \eqref{2026-6-27-3-1} so that $C\varepsilon^2\le \frac{\delta^2}{4}$, we deduce
\begin{align}\label{ping23-3}
\sup_{t\in[0,T]}(\|a\|_{{N}}+\|\u\|_{{N}}+\|\b\|_{{N}})\le \frac{\delta}{2},
\end{align}
which together with  the continuity argument  implies that the local solution
can be extended as a global one in time.
%Moreover, from \eqref{pri-2-2-22-8}, we also have the following decay rate
%\begin{align*}%\label{ping40}\norm{a(t)}{{4r+2}}+\norm{\u(t)}{{4r+2}}+\norm{\b(t)}{{4r+2}}\le  C(1+t)^{-\frac{N-4r-2}{2(2r+2)}}.\end{align*}

Moreover, for any $r+1\leq\gamma< N $,  thanks to the following interpolation inequality
\begin{align*}%\label{ming22}
	\norm{\cdot}{{\gamma}}\le\norm{\cdot}{{r+1}}^{\frac{{N}-\gamma}{{N}-r-1}}
	\norm{\cdot}{{N}}^{\frac{\gamma-r-1}{{N}-r-1}},
\end{align*}
we can conclude by \eqref{bu2026-6-28-2}, \eqref{pri-2-2-22-8} and \eqref{bu-pri-2-2-22-10} that  the   decay rate for the higher order energy
\begin{align*}%\label{ping44}
\norm{a(t)}{{\gamma}}+\norm{\u(t)}{{\gamma}}+\norm{\b(t)}{{\gamma}}\le C(1+t)^{-\frac{N-\gamma}{2r+2}}.
\end{align*}
This completes the proof of Theorem \ref{dingli}.\end{proof}

\bigskip
\section*{Declarations}

%\noindent{\bf Ethical Approval }\

 %We certify that this manuscript is original and has not been published and will not be submitted elsewhere for
%publication while being considered by Mathematische Annalen.

\vskip .1in
\noindent{\bf Competing interests  }\

On behalf of all authors, the corresponding author states that there is no potential conflicts of interest with respect to the research of this article.
\vskip .1in
\noindent{\bf Authors' contributions   }\

This work was carried out in collaboration of three authors. Qiao,
Wu, Xu,  and Zhai proposed the question and presented some ideas of the proof. Xu carried out the  study of the existence and drafted the manuscript.    All authors read and approved the final manuscript.

\vskip .1in
\noindent{\bf Funding   }\

 Wu was partially supported by the National Science Foundation of the United
States under DMS 2104682 and DMS 2309748. Xu was partially supported by   the
National Natural Science Foundation of China 12326430
and the  Natural Science Foundation of Shandong Province ZR2026MS0022.
Zhai was partially supported by the Guangdong Provincial Natural Science
Foundation under grant  2024A1515030115.
\vskip .1in
\noindent{\bf Availability of data and materials  }\

Data and materials  sharing
not applicable to this article as no data and  materials
 were generated or analyzed during the current study.

\end{document}